\documentclass[12pt]{article}
\usepackage[british]{babel}
\usepackage{amsmath}
\usepackage{amsthm}
\usepackage{amsfonts}
\usepackage{amssymb}
\usepackage{fancyhdr}
\usepackage[all]{xy}           
\usepackage{graphicx}
\usepackage[left=3cm, right=3cm]{geometry}
\usepackage{enumerate}
\usepackage{cite}
\usepackage{hyperref} 

\theoremstyle{plain}
\newtheorem{thm}{Theorem}[section]
\newtheorem{cor}[thm]{Corollary}
\newtheorem{lem}[thm]{Lemma}
\newtheorem{pro}[thm]{Proposition}

\newtheorem{problem}[thm]{Problem}

\theoremstyle{definition}
\newtheorem{examp}[thm]{Example}
\newtheorem{defi}{Definition}

\newtheorem{remark}[thm]{Remark}

\def\ad{\text{\rm ad}}
\def\dim{\text{\rm dim}}

\author{Tsiu-Kwen Lee}

\title{Lie ideals and derivations of exceptional prime rings}
\date{}

\begin{document}

\maketitle

\centerline {Department of Mathematics, National Taiwan
University}

\centerline {Taipei, Taiwan}

\centerline {tklee@math.ntu.edu.tw}

\begin{abstract}\vskip6pt
\noindent A prime ring $R$ with extended centroid $C$ is said to be exceptional if both $\text{\rm char}\,R=2$ and $\dim_CRC=4$.
Herstein characterized additive subgroups $A$ of a  nonexceptional simple ring $R$ satisfying $\big[A, [R, R]\big]\subseteq A$.
In 1972 Lanski and Montgomery extended Herstein's theorem to nonexceptional prime rings.
In the paper we first extend Herstein's theorem to arbitrary simple rings.
For the prime case, let $R$ be an exceptional prime ring with center $Z(R)$.
It is proved that if $A$ is a noncentral additive subgroup of $R$  satisfying $\big[A, L\big]\subseteq A$
for some nonabelian Lie ideal $L$ of $R$, then
$\beta Z(R)\subseteq A$ for some nonzero $\beta\in Z(R)$, and either $AC=Ca+C$ for some $a\in A\setminus Z(R)$ with $a^2\in Z(R)$ or $[RC, RC]\subseteq AC$.
Secondly, we study certain generalized linear identities satisfied by Lie ideals and then completely characterize derivations $\delta, d$ of $R$ satisfying $\delta d(L)\subseteq Z(R)$ for $L$ a Lie ideal of $R$.
\end{abstract}

{ \hfill\break \noindent 2020 {\it Mathematics Subject Classification.}\ 16N60, 16W10.


\vskip6pt

\section{Introduction}
Throughout the paper, $R$ always denotes an associative ring, not necessarily with unity, with center $Z(R)$. For $a, b\in R$, let $[a, b]:=ab-ba$, the {\it additive commutator} of $a$ and $b$.
For additive subgroups $A, B$ of $R$, we let $AB$ (resp. $[A, B]$) stand for the additive subgroup of $R$ generated by all elements $ab$ (resp. $[a, b]$) for $a\in A$ and $b\in B$.
Every ring $R$ naturally associates a Lie ring structure with the {\it Lie product} $[x, y]$ of two elements $x, y\in R$.
An additive subgroup $L$ of $R$ is said to be a {\it Lie ideal} of $R$ if $[L, R]\subseteq L$. A Lie ideal $L$ of $R$ is called {\it abelian} if $[L, L]=0$. Otherwise, it is called {\it nonabelian}. Clearly, every additive subgroup $A$ of $R$ is a Lie ideal of $R$ if either $A\subseteq Z(R)$ or $[R, R]\subseteq A$.
A ring $R$ is called {\it simple} if $R^2\ne 0$ and the only ideals of $R$ are $\{0\}$ and itself.  A famous theorem due to Herstein characterizes Lie ideals of simple rings as follows (see \cite[Theorem 1.5]{herstein1969}).

\begin{thm} (Herstein)\label{thm10}
Let $R$ be a simple ring. Then, given a Lie ideal $L$ of $R$, either $[R, R]\subseteq L$ or $L\subseteq Z(R)$ unless $\text{\rm char}\,R=2$ and $\dim_{Z(R)}R=4$.
 \end{thm}

 The following generalization of Theorem \ref{thm10} was also proved by Herstein (see the proof below \cite[Theorem 1.13]{herstein1969}).

\begin{thm} (Herstein)\label{thm20}
Let $R$ be a simple ring. If $A$ is an additive subgroup of $R$ satisfying $\big[A, [R, R]\big]\subseteq A$, then either $A\subseteq Z(R)$ or $[R, R]\subseteq A$ unless $\text{\rm char}\,R=2$ and $\dim_{Z(R)}R=4$.
 \end{thm}

 In \cite{lee2025} the author and Lin characterized Lie ideals of arbitrary simple rings.

\begin{thm} (\cite[Theorem 4.6]{lee2025})\label{thm14}
Let $R$ be a simple ring, and let $L$ be a Lie ideal of $R$. Then
$L\subseteq Z(R)$, $L=Z(R)a+Z(R)$ for some $a\in L\setminus Z(R)$, or $[R, R]\subseteq L$.
\end{thm}

 A ring $R$ is called {\it prime} if, given $a, b\in R$,
 $aRb=0$ implies that either $a=0$ or $b=0$. Given a  prime ring $R$, let $Q_s(R)$ be the Martindale symmetric ring of quotients of $R$. It is known that $Q_s(R)$ is also a prime ring, and its center, denoted by $C$, is a field, which is called the {\it extended centroid} of $R$. The extended centroid $C$ will play an important role in our present study. See \cite{beidar1996} for details.\vskip6pt

\begin{defi}\label{defi1}
A prime ring $R$ is called {\it exceptional} if both $\text{\rm char}\,R=2$ and $\dim_CRC=4$.
Otherwise, it is called {\it nonexceptional}.
\end{defi}

A Lie ideal of a ring $R$ is called {\it proper} if it is of the form $[M, R]$ for some nonzero ideal $M$ of $R$.
Lanski and Montgomery characterized additive subgroups of nonexceptional prime rings, which are invariant under Lie ideals as follows (see \cite[Theorem 13]{lanski1972}).

\begin{thm} (Lanski and Montgomery 1972)\label{thm15}
Let $L$  be a Lie ideal of a nonexceptional prime ring $R$, and let $A$ be an additive subgroup
of $R$. If $[A, L]\subseteq A$, then $L\subseteq Z(R)$, $A\subseteq Z(R)$, or $A$ contains a proper Lie ideal of $R$.
\end{thm}

Due to Theorem \ref{thm15}, we focus on the structure of Lie ideals of exceptional prime rings.
The first purpose of the paper is to study the following.

\begin{problem}\label{problem3}
Let $R$ be an exceptional prime ring, and let $L$ be a nonabelian Lie ideal of $R$. Characterize additive subgroups $A$ of $R$
satisfying $\big[A, L\big]\subseteq A$.
\end{problem}

 The following characterizes Lie ideals of exceptional prime rings.

\begin{lem}\label{lem19}
Let $L$ be a noncentral Lie ideal of an exceptional prime ring $R$.

(i)\ If $L$ is abelian, then $LC=[a, RC]=Ca+C$ for any $a\in L\setminus Z(R)$.

(ii)\ If  $L$ is nonabelian, then $L$ contains a proper Lie ideal of $R$. In this case, $LC$ is equal to either $[RC, RC]$ or $RC$.
\end{lem}

\begin{proof}
(i)\ It follows directly from \cite[Lemmas 6.1 and 4.3]{lee2025}.

(ii)\ By Lemma \ref{lem8} (iv) in the next section, $[I, R]\subseteq L$, where $I:=R[L, L]R$, a nonzero ideal of $R$. Thus $L$ contains a proper Lie ideal of $R$. Since $IC=RC$ and $\dim_CRC=4$, it follows that
either $LC=[RC, RC]$ or $LC=RC$.
\end{proof}

\begin{defi}\label{defi2}
Let $R$ be an exceptional prime ring. We say that a nonabelian Lie ideal $L$ of $R$ is of Type I (resp. Type II) if $LC=[RC, RC]$ (resp. $LC=RC$).
\end{defi}

Due to Lemma \ref{lem19}, Problem \ref{problem3} is solved in Theorem \ref{thm16} for exceptional simple rings, Theorem \ref{thm19} for $L$ nonabelian of Type I and Theorem \ref{thm23}  for $L$ nonabelian of Type II.
In particular, the following are
two main theorems, i.e., Theorems \ref{thm16} and \ref{thm19}.
We will use ``iff" for ``if and only" in the text.

\vskip6pt
\noindent{\bf Theorem A.}
{\it Let $R$ be an exceptional simple ring, and let $A$ be a noncentral additive subgroup of $R$.
Then $\big[A, [R, R]\big]\subseteq A$ iff either $Z(R)\subseteq A\subseteq [R, R]$ or $[R, R]\subseteq A$.}
\vskip6pt

\vskip6pt
\noindent{\bf Theorem B.}
{\it Let $R$ be an exceptional prime ring, and let $A$ be a noncentral additive subgroup of $R$. Suppose that $\big[A, L\big]\subseteq A$, where $L$ is a nonabelian Lie ideal of $R$. Then the following hold:

(i)\ $\beta Z(R)\subseteq A$ for some nonzero $\beta\in Z(R)$;

(ii)\ Either $AC=Ca+C$ for some $a\in A\setminus Z(R)$ with $a^2\in Z(R)$ or $[RC, RC]\subseteq AC$.}
\vskip6pt

In Theorem B (ii), if $[RC, RC]\subseteq AC$, it is natural to ask whether $A$ contains a proper Lie ideal of $R$. This is generally not true even when $AC=[RC, RC]$
(see Example \ref{examp3}).

The study of Lie ideals of prime rings is also rather related to that of linear identities satisfied by Lie ideals (acted by derivations).
By a {\it derivation} of a ring $R$ we mean an additive map $d\colon R\to R$ satisfying $d(xy)=xd(y)+d(x)y$ for all $x, y\in R$.
Given an element $a\in R$, the map $x\mapsto [a, x]$ for $x\in R$ is a derivation, which is called the {\it inner} derivation induced by the element $a$. We denote this derivation by $\text{\rm ad}_a$, that is,
$\text{\rm ad}_a(x)=[a, x]$ for all $x\in R$.
In a prime ring $R$ of characteristic not two, Posner proved: (i) If $\delta$ and $d$ are derivations of $R$ such that $\delta d$ is also a derivation of $R$, then one of $\delta$ and $d$ is zero (\cite[Theorem 1]{posner1957}); (ii) If $d$ is a nonzero derivation of $R$ such that $[d(x), x]\in Z(R)$ for all $x\in R$, then $R$ is commutative (\cite[Theorem 2]{posner1957}). For the case of Lie ideals, we refer the reader to \cite{bergen1981, lee1981, lee1983, ke1985}. In particular, we mention the following.

\begin{thm}\label{thm26}
Let $R$ be a prime ring with nonzero derivations $\delta$ and $d$, and let $L$ be a Lie ideal of $R$. Suppose that $\delta d(L)\subseteq Z(R)$.

(i) If $\text{\rm char}\,R\ne2$, then $L\subseteq Z(R)$ (\cite[Theorem 4]{lee1983});

(ii) If $\text{\rm char}\,R=2$ and $\dim_CRC>4$, then either $L\subseteq Z(R)$ or $\delta=\lambda d$ some some $\lambda\in C$ (\cite[Main Theorem, p.274]{ke1985}).
\end{thm}

In Theorem \ref{thm26}, if both $\delta$ and $d$ are inner derivations, i.e., $\delta=\text{\rm ad}_a$ and $d=\text{\rm ad}_b$ for some $a, b\in R$, then
$\delta d(L)\subseteq Z(R)$ iff the central linear identity $\big[a, [b, x]\big]\in Z(R)$ holds for all $x\in L$.

\begin{problem}\label{problem2}
Let $R$ be an exceptional prime ring, and let $L$ be a Lie ideal of $R$. Characterize elements $a_1,\ldots,a_n\in R$ satisfying
$\big[a_1, [a_2, [\cdots, [a_n, L]\cdots]]\big]\subseteq Z(R)$.
\end{problem}

In view of Lemma \ref{lem19}, Problem \ref{problem2} is solved as follows, i.e., Theorem \ref{thm36}.
\vskip6pt

\noindent{\bf Theorem C.}
{\it Let $R$ be an exceptional prime ring, and let $L$ be a noncentral Lie ideal of $R$, and $a_i\in R$, $1\leq i\leq n$.

(i)\ If $L$ is either abelian or nonabelian of Type I, then $\big[a_1, a_2,\cdots,a_n, L\big]\subseteq Z(R)$ iff one of $a_1,\ldots,a_n$ lies in $[RC, RC]$.

(ii)\ If $L$ is nonabelian of Type II, then $\big[a_1, a_2,\cdots,a_n, L\big]\subseteq Z(R)$ iff either $a_n\in Z(R)$ or one of $a_1,\ldots,a_{n-1}$ lies in $[RC, RC]$.}
\vskip6pt

In \cite{lanski2010} Lanski described an interesting class of expressions,
based on commutators of these Engel-type polynomials that force any $R$ satisfying them to satisfy $S_4$, the standard identity of degree $4$.

\begin{defi}\label{defi3}
Let $E_1(x) = x$ and for $k\geq 1$
$$
E_{k+1}(x_1,\ldots,x_{k+1}):=\big[E_{k}(x_1,\ldots,x_{k}), x_{k+1}\big]
$$
for $x, x_1,\ldots, x_{k+1}\in R$.
\end{defi}

Let $R$ be a prime ring with $L$ a nonabelian Lie ideal. Lanski proved that if
\begin{eqnarray}
\big[E_m(x_1,\ldots,x_m), E_k(y_1,\ldots,y_k)\big]=0
\label{eq:123}
\end{eqnarray}
for all $x_i, y_j\in L$, then $R$ is exceptional (see \cite[Theorem 2]{lanski2010}).
We let $E_m(L)^+$ be the additive subgroup of $R$
generated by all elements $E_m(x_1,\ldots,x_m)$ for all $x_i\in L$.
Thus
Eq.\eqref{eq:123} holds
for all $x_i, y_j\in L$ iff $\big[E_m(L)^+, E_k(L)^+\big]=0$.
We completely characterize Eq.\eqref{eq:123} as follows (i.e., Theorem \ref{thm37}).\vskip6pt

\noindent{\bf Theorem D.}
{\it Let $R$ be a prime ring, and let $L$ be a noncentral Lie ideal of $R$, and either $m>1$ or $ k>1$. Then $\big[E_m(L)^+, E_k(L)^+\big]=0$ iff $R$ is exceptional and $L\subseteq [RC, RC]$.}
\vskip6pt

We come back to Theorem \ref{thm26}. Let $R$ be a prime ring. It is well-known that every derivation $d$ of $R$ can be uniquely extended to a derivation, denoted by $d$ also, of $Q_s(R)$ such that $d(C)\subseteq C$. Thus the derivation $d$ of $R$ is also extended to a derivation of $RC$.
As the second purpose of the paper, we study Theorem \ref{thm26} in the context of exceptional prime rings. \vskip6pt

\begin{problem}\label{problem1}
Let $R$ be an exceptional prime ring, and let $L$ be a noncentral Lie ideal of $R$. Characterize derivations $\delta, d$ of $R$ satisfying $\delta d(L)\subseteq Z(R)$.
\end{problem}

We let $\Bbb Z_2[t]$ stand for the polynomial ring in the variable $t$ over $\Bbb Z_2$, the Galois field of two elements.

\begin{remark}\label{remark1}
Let $R$ be an exceptional prime ring with extended centroid $C$, and let $L$ be a noncentral Lie ideal of $R$.

(i)\ If $L$ is either abelian or nonabelian of Type I, then $LC$ and $L$  satisfy the same differential identities (with coefficients in $RC$).
Firstly, let $L$ be abelian. Then, by \cite[Lemma 4.1]{lee2025}, $LC=[a, RC]$ for some $a\in L\setminus Z(R)$. Then $[a, R]\subseteq L\subseteq [a, RC]=LC$. Since $R$ and $RC$
 satisfy the same differential identities (see \cite[Theorem 2]{lee1992}), it follows that  $LC$ and $L$  satisfy the same differential identities. Secondly, let $L$ is nonabelian of Type I. In view of Lemma \ref{lem8} (iv), $[I, R]\subseteq L\subseteq LC=[RC, RC]$. Since $I$ and $RC$
 satisfy the same differential identities (see \cite[Theorem 2]{lee1992}), it follows that  $LC$ and $L$  satisfy the same differential identities.

 (ii) When $L$ is nonabelian of Type II, $L$ and $LC$ do not in general satisfy the same differential identities. Indeed, let $R:=\text{\rm M}_2(\Bbb Z_2[t])$, and let $L:=[R, R]+\Bbb Z_2e_{11}$. In this case, $C=\Bbb Z_2(t)$, $RC=\text{\rm M}_2(C)$, and $L$ is a nonabelian Lie ideal of Type II. Let $d$ be an X-outer derivation of $R$ such that $d(e_{11})=0$. Then $d(L)\subseteq [R, R]$ and so
 $$
 \big[\big[d(x), [y_1, y_2]\big], z\big]=0
 $$
 for all $x\in L$ and $y_1, y_2, z\in R$. Since $d$ is X-outer, $d(\beta)\ne 0$ for some $\beta\in C$. Then
 $$
  \big[\big[d(\beta e_{11}), [e_{11}, e_{12}]\big], e_{11}\big]= \big[\big[d(\beta)e_{11}, [e_{11}, e_{12}]\big], e_{11}\big]=d(\beta)e_{12}\ne 0,
 $$
 as desired. Finally, we choose a derivation $d$ of $R$ satisfying the required conditions.
 Let $R:=\text{\rm M}_2(\Bbb Z_2[t])$.
 Given $f(t)\in C=\Bbb Z_2(t)$, let $f'(t)$ be the usual derivative of $f(t)$, and then
 $$
 d\big(\sum_{1\leq i,  j\leq 2}f_{ij}(t)e_{ij}\big):=\sum_{1\leq i,  j\leq 2}f'_{ij}(t)e_{ij}.
 $$
 Then $d$ is clearly an X-outer derivation of $R$, $d(e_{11})=0$, and $d(t)=1$.
\end{remark}

In view of Lemma \ref{lem19} and Remark \ref{remark1}, Problem \ref{problem1} is solved in Theorem \ref{thm34} for $L$ noncentral abelian, Theorem \ref{thm29} for $L$ nonabelian of Type I, and Theorem \ref{thm35} for $L$ nonabelian of Type II.
In particular, the following is Theorem \ref{thm29}.\vskip6pt

\noindent{\bf Theorem E.}
{\it Let $R$ be an exceptional prime ring with nonzero derivations $\delta, d$, and let $I$ be a nonzero ideal of $R$.

(i)\ If $d$ is X-inner and $\delta$ is X-outer, then $\delta d([I, I])\subseteq Z(R)$ iff $d([R, R])\subseteq Z(R)$;

(ii)\ If both $d$ and $\delta$ are X-inner, then $\delta d([I, I])\subseteq Z(R)$ iff either $d([R, R])\subseteq Z(R)$ or $\delta([R, R])\subseteq Z(R)$;

(iii)\ If $d$ is X-outer and $\delta$ is X-inner, then $\delta d([I, I])\subseteq Z(R)$ iff $\delta([R, R])\subseteq Z(R)$;

(iv)\ If both $\delta$ and $d$ are X-outer, then $\delta d([I, I])\subseteq Z(R)$ iff there exists $\beta\in Z(R)$ such that neither $d(\beta)=0$ nor $\delta(\beta)=0$, and
\begin{eqnarray*}
\delta(\beta)d+d(\beta)\delta=\text{\rm ad}_g\ \text{\rm and}\ \ d^2=\mu d+ \text{\rm ad}_h
\end{eqnarray*}
for some $g\in RC$, $h\in [RC, RC]$ and $\mu\in C$ such that $\mu=0$ iff $g\in [RC, RC]$, and if $\mu\ne 0$ then
$
g^2+\delta(\beta)\mu g\in C.
$}

\section{Preliminaries}
Let $A$ be a subset of a ring $R$. We denote by $\overline{A}$ the subring of $R$ generated by $A$.

\begin{lem}\label{lem8}
(i)\ Let $R$ be a a prime ring. If $\big[a, [I, I]\big]=0$ where $a\in R$ and $I$ is a nonzero ideal of $R$, then $a\in Z(R)$;

(ii)\ Let $R$ be a a prime ring. If $\big[a, R\big]\subseteq Z(R)$ where $a\in R$, then $a\in Z(R)$;

(iii)\ $[I, A]=[I, \overline A]$ for  any ideal $I$ of $R$ and any additive subgroup $A$ of $R$;

(iv)\ If $L$ is a Lie ideal of a ring $R$, then $I:=R[L, L]R\subseteq L+L^2$ and $[I, R]\subseteq L$;

(v)\ Let $R$ be a prime ring with Lie ideals $K, L$. If $[K, L]\subseteq Z(R)$, then one of $K$  and $L$ is central except when $R$ is exceptional (see \cite[Lemma 7]{lanski1972}).
\end{lem}

For (i) we refer the reader to \cite[Lemma 1.5]{herstein1969} with the same proof.  For (ii), if $\big[a, R\big]\subseteq Z(R)$, then applying the Jacobi identity $\big[a, [x, y]\big]=\big[[a, x], y]\big]+\big[x, [a, y]\big]$ for $a, x, y\in R$, we have
$$
\big[a, [R, R]\big]\subseteq \big[[a, R], R\big]+\big[R, [a, R]\big]=0
$$
and so, by (i), $a\in Z(R)$. See \cite[Fact 2]{eroglu2019} for (iii).
For (iv), it follows from \cite[Lemma 2.1]{lee2022} that $I:=R[L, L]R\subseteq L+L^2$.
By (iii), we get
$$
[I, R]\subseteq [L+L^2, R]\subseteq [\overline L, R]=[L, R]\subseteq L,
$$
as desired.

\begin{lem} (\cite[Lemma 2.3]{lee2022})\label{lem20}
If $R$ is a prime ring, then $0\ne \big[[R, R], [R, R]\big]\subseteq Z(R)$ iff $R$ is exceptional.
\end{lem}

\begin{lem}\label{lem7}
Let $R$ be an algebra over a field $K$ satisfying $\dim_KR/[R, R]\leq 1$, and let $A$ be an additive subgroup of $R$.
If $\big[A, [R, R]\big]\subseteq A$, then either $A\subseteq [R, R]$ or $[A, R]\subseteq A$.
\end{lem}

\begin{proof}
Let $a\in A$. If $a\notin [R, R]$, it follows from $\dim_FR/[R, R]\leq 1$ that $R= [R, R]+Ka$ and so
$$
[a, R]=\big[a, [R, R]+Ka\big]=\big[a, [R, R]\big]\subseteq A.
$$
Therefore, $A$ is the union of its two additive subgroups:\ $\{a\in A\mid a\in [R, R]\}$ and $\{a\in A\mid [a, R]\}\subseteq A$.
Since $A$ cannot be the union of its two proper subgroups, this implies that either $A\subseteq [R, R]$ or $[A, R]\subseteq A$.
\end{proof}

It is well-known that $\dim_{Z(R)}R/[R, R]=1$ for any finite-dimensional central simple algebra $R$.

\begin{lem}\label{lem6}
Let $R$ be an exceptional prime ring, and $a\in R$. Then $a\in [RC, RC]$ iff $a^2\in Z(R)$.
\end{lem}

\begin{proof}
Let $F$ be the algebraic closure of $C$.
Then there exists an $F$-algebra isomorphism $\phi\colon RC\otimes_CF\to\text{\rm M}_2(F)$.
Let $\text{\rm tr}(x)$ denote the trace of $x\in \text{\rm M}_2(F)$. Clearly, given $x\in \text{\rm M}_2(F)$, $x^2\in F$ iff $\text{\rm tr}(x)=0$ as $\text{\rm char}\,F=2$.

Let $a\in [RC, RC]$. Then $\phi(a)\in [\phi(RC), \phi(RC)]$ and so $\text{\rm tr}(\phi(a))=0$. This implies that $\phi(a^2)=\phi(a)^2\in F$ and so
$a^2\in Z(R)$. Conversely, let $a^2\in Z(R)$. Then $\phi(a)^2\in F$ and so $\text{\rm tr}(\phi(a))=0$. Thus there exist finitely many $x_i, y_i\in \text{\rm M}_2(F)$ such that
$
\phi(a)=\sum_i[x_i, y_i]
$
and so $a=\sum_i[\phi^{-1}(x_i), \phi^{-1}(y_i)]$. This implies that $a\in [RC, RC]$, as desired.
\end{proof}

In Lemma \ref{lem6}, if $a^2\in Z(R)$, $a\in [R, R]$ is generally not true.

\begin{examp}\label{examp2}
Let $R:=\text{\rm M}_2(t{\Bbb Z}_2[t])$, and let
$
a:=\left (
      \begin{array}{cc}
    t        & t\\
     t   & t
      \end{array} \right)\in R$.
Then $C={\Bbb Z}_2(t)$, $RC=\text{\rm M}_2(C)$, $a^2=0$, $a\in [RC, RC]$ but it is clear that $a\notin [R, R]$.
\end{examp}

The following lemma almost has the same proof as that of
\cite[Lemma 4.3]{calugareanu2024}.

\begin{lem}
Let $R$ be a prime ring, and let $I$ be a nonzero ideal of $R$. If $a\in
RC\setminus C$, then $\dim_C\big[a, IC\big]>1$.
\label{lem11}
\end{lem}

\noindent {\bf The Skolem-Noether Theorem.}\ {\it Let $R$ be a finite-dimensional central simple algebra. Then every derivation $d$ of $R$ satisfying $d(Z(R))=0$ is inner.}\vskip6pt

\section{Problem \ref{problem3}: Simple rings}
We are now ready to prove the first main theorem, i.e., Theorem A, in the paper.

\begin{thm}\label{thm16}
Let $R$ be an exceptional simple ring, and let $A$ be a noncentral additive subgroup of $R$.
Then $\big[A, [R, R]\big]\subseteq A$ iff either $Z(R)\subseteq A\subseteq [R, R]$ or $[R, R]\subseteq A$.
\end{thm}

\begin{proof}
In this case, $C=Z(R)$, $RC=R$, and $\dim_{Z(R)}R/[R, R]=1$.

``$\Rightarrow$'':\ In view of Lemma \ref{lem7}, either $A\subseteq [R, R]$ or $[A, R]\subseteq A$.
The latter case implies that $A$ is a Lie ideal of $R$ and so, by Theorem \ref{thm14}, either $A=Z(R)a+Z(R)$ for some $a\in A\setminus Z(R)$ or $[R, R]\subseteq A$. However, if $A=Z(R)a+Z(R)$ for some $a\in A\setminus Z(R)$, then, by Lemma \ref{lem11},
$
A=[a, R]=Z(R)a+Z(R).
$
Hence $Z(R)\subseteq A\subseteq [R, R]$.

Assume next that $A\subseteq [R, R]$. Since $A$ is not central, it follows from Lemma \ref{lem8} (i) that $0\ne \big[A, [R, R]\big]\subseteq A\subseteq [R, R]$.
By Lemma \ref{lem20}, we get
$$
0\ne \big[A, [R, R]\big]\subseteq \big[[R, R], [R, R]\big]\subseteq Z(R).
$$
Since $Z(R)$ is a field and $RZ(R)=R$,
$$
Z(R)=Z(R)\big[A, [R, R]\big]=\big[A, [Z(R)R, R]\big]=\big[A, [R, R]\big]\subseteq A.
$$
Hence $Z(R)\subseteq A\subseteq [R, R]$, as desired.

``$\Leftarrow$'':\ If $[R, R]\subseteq A$, then $\big[A, [R, R]\big]\subseteq [R, R]\subseteq A$, as desired. Assume next that
$Z(R)\subseteq A\subseteq [R, R]$. Since $R$ is exceptional, by Lemma \ref{lem20} we get
$$
\big[A, [R, R]\big]\subseteq \big[[R, R], [R, R]\big]\subseteq Z(R)\subseteq A,
$$
as desired
\end{proof}

In Theorem \ref{thm16}, let $A:={\Bbb Z}_2a+Z(R)$, where $a\in [R, R]\setminus Z(R)$. Then $\big[A, [R, R]\big]\subseteq A$ but the additive subgroup $A$
is not a $Z(R)$-subspace of $R$ if $Z(R)\ne {\Bbb Z}_2$. The following is an immediate consequence of Theorem \ref{thm16}.

\begin{cor}\label{cor1}
Let $R$ be an exceptional simple ring, and let $A$ be a noncentral $Z(R)$-subspace of $R$.
Then $\big[A, [R, R]\big]\subseteq A$ iff $A$ is equal to $Z(R)+Z(R)w$ for some $w\in [R, R]\setminus Z(R)$, $[R, R]$, or $R$.
In this case, $A$ is a Lie ideal of $R$.
\end{cor}

We also need the following well-known theorem due to Herstein (see the corollary of \cite[Theorem 1.5]{herstein1969}).

\begin{thm} (Herstein) \label{thm24}
If $R$ is a noncommutative simple ring, then $\overline{[R, R]}=R$.
\end{thm}

\begin{lem}\label{lem3}
Let $R$ be an exceptional simple ring, and let $A$ be a noncentral subring of $R$. Then $Z(R)\subseteq A\subseteq [R, R]$ iff $A=Z(R)a+Z(R)$ for some $a\in A\setminus Z(R)$ with $a^2\in Z(R)$.
\end{lem}

\begin{proof}
``$\Rightarrow$'': Assume that $Z(R)\subseteq A\subseteq [R, R]$. In view of Theorem \ref{thm16}, $\big[A, [R, R]\big]\subseteq A$ and $AZ(R)=A$ as $A$ is a subring of $R$.  By Corollary \ref{cor1}, $A$ is equal to either $Z(R)+Z(R)w$ for some $w\in [R, R]\setminus Z(R)$ or $[R, R]$.
 If $A=[R, R]$, then, by Theorem \ref{thm24}, $R=\overline {[R, R]}\subseteq A$, a contradiction. Thus $A=Z(R)a+Z(R)$ for some $a\in A\setminus Z(R)$. Since $A\subseteq [R, R]$, it follows that $a^2\in Z(R)$ (see Lemma \ref{lem6}).

``$\Leftarrow$'':\ Assume that $A=Z(R)a+Z(R)$ for some $a\in A\setminus Z(R)$ with $a^2\in Z(R)$. Since $R$ is an exceptional simple ring, it is clear that
$a\in [R, R]$ as $a^2\in Z(R)$ (see Lemma \ref{lem6}). Therefore, $Z(R)\subseteq A\subseteq [R, R]$.
\end{proof}

\begin{cor}\label{cor6}
Let $R$ be an exceptional simple ring, and let $A$ be a noncentral subring of $R$.
Then $\big[A, [R, R]\big]\subseteq A$ iff either $A=R$ or $A=Z(R)a+Z(R)$ for some $a\in A\setminus Z(R)$ with $a^2\in Z(R)$.
\end{cor}

\begin{proof}
``$\Rightarrow$'':\ Since $R$ is exceptional and $\big[A, [R, R]\big]\subseteq A$, it follows from Theorem \ref{thm16} that either $Z(R)\subseteq A\subseteq [R, R]$ or $[R, R]\subseteq A$. The latter case implies that $R=\overline{[R, R]}\subseteq A$ and so $A=R$ (see Theorem \ref{thm24}). Assume next that $Z(R)\subseteq A\subseteq [R, R]$.
By Lemma \ref{lem3}, $A=Z(R)a+Z(R)$ for some $a\in A\setminus Z(R)$ with $a^2\in Z(R)$.

``$\Leftarrow$'':\ Assume that $A=Z(R)a+Z(R)$ for some $a\in A\setminus Z(R)$ with $a^2\in Z(R)$. Since $a^2\in Z(R)$ and $A$ is noncentral, we get $a\in [R, R]\setminus Z(R)$ (see Lemma \ref{lem6}). Thus $Z(R)\subseteq A\subseteq [R, R]$ and so, by Theorem \ref{thm16},
$
\big[A, [R, R]\big]\subseteq A,
$
as desired.
\end{proof}

\section{Problem \ref{problem3}: Prime rings}

In this section we turn to the proof of Theorem B.

\begin{pro}\label{pro1}
Let $R$ be an exceptional prime ring, and let $A$ be a noncentral additive subgroup of $R$. Suppose that $\big[A, [I, I]\big]\subseteq A$, where $I$ is a nonzero ideal of $R$. Then Then the following hold:

(i)\ $\beta Z(R)\subseteq A$ for some nonzero $\beta\in Z(R)$;

 (ii)\ Either $AC=Ca+C$ for some $a\in A\setminus Z(R)$ with $a^2\in Z(R)$ or $[RC, RC]\subseteq AC$.
\end{pro}

\begin{proof}
(i)\ By the fact that $RC$ is a central simple algebra over $C$, it follows that $IC=RC$. Since
$
\big[A, [I, I]\big]\subseteq A,
$
we have
$\big[AC, [IC, IC]\big]\subseteq AC$, that is,
$
\big[AC, [RC, RC]\big]\subseteq AC.
$
Then
$$
S:=\big[[I, I], \big[A, [I, I]\big]\big]\subseteq \big[[I, I], A\big]\subseteq A.
$$
It follows from Lemma \ref{lem20} that $S\subseteq Z(R)\cap A$.

Suppose first that $S\ne \{0\}$. Then there exists a nonzero $\beta\in S\cap A\subseteq Z(R)$ and so
\begin{eqnarray*}
\beta Z(R)&\subseteq& \big[[I, I], \big[A, [I, I]\big]\big]Z(R)\\
                   &\subseteq& \big[[I, I], \big[A, [I, IZ(R)]\big]\big]\\
                   &\subseteq& \big[[I, I], \big[A, [I, I]\big]\big]\\
                   &\subseteq& A.
\end{eqnarray*}

Suppose next that $S=\{0\}$. It follows from Lemma \ref{lem8} (i) that
$
\big[A, [I, I]\big]\subseteq A\cap Z(R).
$
Since $A$ is not central, $\big[A, [I, I]\big]\ne 0$ (see Lemma \ref{lem8} (i)) and so  there exists a nonzero $\beta\in \big[A, [I, I]\big]\subseteq A\cap Z(R)$ and hence $\beta Z(R)\subseteq A$, as desired.

(ii)\ Note that
$
\big[AC, [RC, RC]\big]=\big[AC, [IC, IC]\big]\subseteq AC.
$
By Theorem \ref{thm16}, it follows that either $AC\subseteq [RC, RC]$ or $[RC, RC]\subseteq AC$. That is, either $AC\subsetneq [RC, RC]$ or $[RC, RC]\subseteq AC$. If $AC\subsetneq [RC, RC]$, then $\dim_CAC=2$ and so $AC=Ca+C$ for some $a\in A\setminus Z(R)$. Since $a\in [RC, RC]$, we get $a^2\in Z(R)$ (see Lemma \ref{lem6}).
\end{proof}

We are now ready to prove the second main theorem, i.e., Theorem B.

\begin{thm}\label{thm19}
Let $R$ be an exceptional prime ring, and let $A$ be a noncentral additive subgroup of $R$. Suppose that $\big[A, L\big]\subseteq A$, where $L$ is a nonabelian Lie ideal of $R$. Then the following hold:

(i)\ $\beta Z(R)\subseteq A$ for some nonzero $\beta\in Z(R)$;

(ii)\ Either $AC=Ca+C$ for some $a\in A\setminus Z(R)$ with $a^2\in Z(R)$ or $[RC, RC]\subseteq AC$.
\end{thm}

\begin{proof}
It follows from Lemma \ref{lem8} (iv) that $I:=R[L, L]R\subseteq L+L^2$ and $[I, I]\subseteq L$. Thus
$
\big[A, [I, I]\big]\subseteq A.
$
We are done by Proposition \ref{pro1}.
\end{proof}

In Theorem \ref{thm19} (ii) we cannot conclude that $A$ contains a proper Lie ideal of $R$ even when $AC=[RC, RC]$.

\begin{examp}\label{examp3}
{\rm Let $R:=\text{\rm M}_2(\Bbb Z_2[t])$, and let
$$
A:=\Bbb Z_2[t](e_{11}+e_{22})+\Bbb Z_2e_{12}+\Bbb Z_2e_{21},
$$
which is an additive subgroup of $R$.
Then $C=\Bbb Z_2(t)$, $RC=\text{\rm M}_2(\Bbb Z_2(t))$, and
$$
[RC, RC]=\Bbb Z_2(t)(e_{11}+e_{22})+\Bbb Z_2(t)e_{12}+\Bbb Z_2(t)e_{21}=AC.
$$
Clearly, we have
$
\big[A, [R, R]\big]=\Bbb Z_2[t](e_{11}+e_{22})\subseteq A.
$
Since every ideal of $R$ is of the form $\text{\rm M}_2(g(t)\Bbb Z_2[t])$ for some $g(t)\in \Bbb Z_2[t]$, it follows that
$A$ contains no any proper Lie ideal of $R$.
}
\end{examp}

We next characterize additive subgroups $A$ of an exceptional prime ring $R$ satisfying $[A, I]\subseteq A$ for some nonzero ideal $I$ of $R$ (see Theorem \ref{thm23} below).

\begin{lem}(\cite[Lemma 2.2 (i)]{abdioglu2017})\label{lem2}
Let $R$ be a prime ring, $a, b\in R$.
Suppose that $\big[a, [b, R]\big]=0$. If $b\notin C$, then $a\in Cb+C$.
\end{lem}

\begin{lem}\label{lem5}
Let $R$ be a prime ring with a noncentral subring $A$. Suppose that $\big[A, I\big]\subseteq A$, where $I$ is a nonzero ideal of $R$.
Then either $A$ contains a nonzero ideal of $R$ or $AC=Ca+C$ for some $a\in A\setminus Z(R)$ with $a^2\in Z(R)$.
\end{lem}

\begin{proof}
Case 1:\ $[A, A]=0$. Then $[A, [A, I]]=0$. In view of \cite[Theorem 2]{chuang1988}, $I$ and $R$ satisfy the same GPIs. Thus $[A, [A, R]]=0$.
Choose an element $a\in A\setminus Z(R)$. Then $\big[x, [a, R]\big]=0$ for all $x\in A$.
In view of Lemma \ref{lem2}, there exist $\alpha, \beta\in C$ such that $x=\alpha a+\beta$. Hence $A\subseteq Ca+C$.
On the other hand, $[a, IC]\subseteq AC$.
By Lemma \ref{lem11}, $\dim_C[a, IC]>1$ and hence $AC=Ca+C$.
In particular, $\big[a, [a, I]\big]=0$, implying that, by \cite[Theorem 1]{posner1957}, $\text{\rm char}\,R=2$ and $a^2\in C$, as desired.

Case 2:\ $[A, A]\ne 0$. For $a, b\in A$ and $x\in I$, we have $[a, b]x=-b[a, x]+[a, bx]$.
Thus $[A, A]I\subseteq A[A, I]+[A, AI]\subseteq A$. Similarly, $I[A, A]\subseteq A$ and hence
Let $\rho:=[A, A]I$ be a nonzero right ideal of $R$ contained in $A$.
Then
$$
0\ne I\rho\subseteq [\rho, I]+\rho I\subseteq [A, I]+\rho\subseteq A,
$$
as desired.
\end{proof}

\begin{lem}\label{lem10}
Let $R$ be a ring with ideals $I$ and $J$. If either $IJ\ne 0$ or $JI\ne 0$, then $[I, J]$ contains a proper Lie ideal of $R$.
\end{lem}

\begin{proof}
Let $i\in I$, $j\in J$ and $r\in R$. Then
$
[ij, r]=[i, jr]+[j, ri],
$
and so
$$
[IJ, R]\subseteq [I, JR]+[J, RI]\subseteq [I, J].
$$
Similarly,
$
[JI, R]\subseteq [I, J].
$
Since either $IJ\ne 0$ or $JI\ne 0$, this proves that $[I, J]$ contains a proper Lie ideal of $R$.
\end{proof}

\begin{thm}\label{thm23}
Let $R$ be an exceptional prime ring, and let $A$ be a noncentral additive subgroup of $R$. Suppose that $\big[A, I\big]\subseteq A$, where $I$ is a nonzero ideal of $R$.
Then either $A$ contains a proper Lie ideal of $R$ or $AC=Ca+C$ for some $a\in A\setminus Z(R)$ with $a^2\in Z(R)$.
\end{thm}

\begin{proof}
Since $\big[A, I\big]\subseteq A$, we get $\big[\overline A, I\big]\subseteq \overline A$. In view of Lemma \ref{lem5}, either $\overline A$ contains a nonzero ideal, say $M$, of $R$ or $\overline AC=Ca+C$ for some $a\in A\setminus Z(R)$ with $a^2\in Z(R)$.

For the latter case, we get $AC\subseteq \overline AC=Ca+C$ and so $[a, IC]\subseteq AC$, that is, $[a, RC]\subseteq AC$ as $IC=RC$.
On the other hand, by Lemma \ref{lem11}, $\dim_C[a, RC]>1$ and hence $AC=Ca+C$. This implies that $[a, RC]\subseteq Ca+C$ and so $[a, RC]= Ca+C$.
Hence $\big[a, [a, RC]\big]=0$, implying $a^2\in C$.

Assume next that $\overline A$ contains a nonzero ideal $M$ of $R$. Then $ [\overline A, I]=[A, I]$ (see Lemma \ref{lem8} (iii)) and so
$$
[M, I]\subseteq [\overline A, I]=[A, I]\subseteq A,
$$
as desired. By the primeness of $R$, we get $MI\ne 0$. In view of Lemma \ref{lem10}, $[M, I]$ contains a proper Lie ideal of $R$, so does $A$.
\end{proof}

\begin{examp}\label{examp4} There exists a Lie ideal $L$ of an exceptional prime ring $R$ such that $L$ is nonabelian of Type II but contains no nonzero ideals of $R$. Indeed, let $R:=\text{\rm M}_2({\Bbb Z}_2[t])$, and let $L:=[R, R]+{\Bbb Z}_2e_{11}$.
In this case, $C={\Bbb Z}_2(t)$ and $RC=\text{\rm M}_2({\Bbb Z}_2(t))$.
Then $L$ is a Lie ideal of $R$ satisfying $LC=RC$. Thus $L$ is nonabelian of Type II.
We claim that $L$ contains no nonzero ideals of $R$. Otherwise, assume that $L$ contains a nonzero ideal $I$ of $R$.
Then there exists a nonzero polynomial $g(t)\in {\Bbb Z}_2[t]$ such that $I=\text{\rm M}_2(g(t){\Bbb Z}_2[t])$.
Then $g(t)e_{11}\in I\subseteq L$, implying that $g(t)e_{11}=z+e_{11}$ for some $z\in [R, R]$. Thus $g(t)=1$. This implies that
$I=R$ and so $L=R$, a contradiction.
\end{examp}

\section{Problem \ref{problem2}}
Let $R$ be a prime ring with extended centroid $C$.
A map $\phi\colon RC\to RC$ is called a {\it generalized linear map} on $RC$ if there exist finitely many $a_i, b_i\in RC$ such that
$
\phi(x)=\sum_ia_ixb_i
$
for $x\in RC$. Let
$
\mathfrak{L}(RC)
$
denote the set of all generalized linear maps on $RC$. Clearly, if $\phi, \eta\in \mathfrak{L}(RC)$, then $\phi+\eta, \phi\eta\in \mathfrak{L}(RC)$, where $\phi\eta$ is the composition of $\phi$ and $\eta$, that is, $\phi\eta(x):=\phi(\eta(x))$ for $x\in RC$.
Thus $\mathfrak{L}(RC)$ forms an algebra over $C$. In fact, applying Martindale's theorem (see \cite[Theorem 2]{martindale1969}) we have
$$
RC\otimes_C(RC)^{op}\cong \mathfrak{L}(RC)
$$
via the canonical map
$$
\sum_ia_i\otimes b_i\mapsto \phi,
$$
where $\phi(x):=\sum_ia_ixb_i$ for $x\in RC$. Note that, in $RC\otimes_C(RC)^{op}$, $\sum_ia_i\otimes b_i=0$ iff $\sum_ib_i\otimes a_i=0$.
Hence, given $\phi\in \mathfrak{L}(RC)$, where $\phi(x)=\sum_ia_ixb_i$ for $x\in RC$, we can define $\phi^*$ by
$$
\phi^*(x)=\sum_ib_ixa_i
$$
for $x\in RC$. Clearly, the map $*$ on $\mathfrak{L}(RC)$ is well-defined. Moreover, given $\phi, \eta\in \mathfrak{L}(RC)$, we have
$$
(\phi+\eta)^*=\phi^*+\eta^*, \ \ (\phi^*)^*=\phi, \ \text{\rm and}\ \ (\phi\eta)^*=\eta^*\phi^*.
$$
That is, $*$ is an involution on $\mathfrak{L}(RC)$.
For $a_1,\ldots,a_n, x\in RC$, let
$$
[a_1,\ldots,a_n, x]:=\ad_{a_1}\cdots\ad_{a_n}(x).
$$
Clearly,
$$
\ad_{a_1}^*=-\ad_{a_1}\ \text{\rm and}\ \ (\ad_{a_1}\cdots\ad_{a_n})^*=(-1)^n\ad_{a_n}\cdots\ad_{a_1}.
$$
Moreover, let $\delta$ be a derivation of $R$. Then
$$
\delta\,\ad_b=\ad_{\delta(b)}+\ad_b\delta
$$
for all $b\in RC$.
We need a preliminary result, which is a special case of \cite[Lemma 2]{kharchenko1978}.\vskip 4pt

\begin{pro} (Kharchenko) \label{pro3}
Let $R$ be a prime ring with derivations $\delta$ and $d$, and $\phi, \varphi, \eta\in \mathfrak{L}(RC)$.

(i)\ If $\delta$ is X-outer, then $\phi\delta+\varphi =0$ implies that $\phi=0$ and $\varphi=0$.

(ii)\ If $\delta$ and $d$ are $C$-independent modulo X-inner derivations, then $\phi\delta+\varphi d+\eta=0$ implies that
$\phi=0$, $\varphi=0$ and $\eta=0$.
\end{pro}

\begin{thm}\label{thm32}
Let $R$ be a prime ring, $\phi\in \mathfrak{L}(RC)$.

(i)\
$
 \phi([x, y])=0
$
for all $x, y\in RC$ iff $\phi^*(y)\in C$ for all $y\in RC$.

(ii)\ If $d\phi=0$ for some X-derivation $d$, then $\phi=0$.
\end{thm}

\begin{proof}
Since $\phi\in \mathfrak{L}(RC)$, there exist finitely many $a_i, b_i\in RC$ such that $\phi(x)=\sum_ia_ixb_i$ for $x\in RC$.

(i)\ We have
\begin{eqnarray*}
\phi([x, y])=0\ \forall x, y\in RC             &\Leftrightarrow& \sum_ia_i[x, y]b_i=0\ \forall x, y\in RC\\
                                                                      &\Leftrightarrow& \phi\,\ad_x=0\ \forall x\in RC \\
                                                                      &\Leftrightarrow&  (\phi\,\ad_x)^*=0\ \forall x\in RC \\
                                                                      &\Leftrightarrow&  -\ad_x\,\phi^*=0\ \forall x\in RC \\
                                                                      &\Leftrightarrow&  \sum_ib_iya_i\in C\ \ \forall y\in R\\
                                                                      &\Leftrightarrow&   \phi^*(y)\in C\ \ \forall y\in RC.
\end{eqnarray*}

(ii)\ Since $d\phi=0$, we have
$$
d\phi(x)=\sum_i d(a_i)xb_i+\sum_ia_i d(x)b_i+\sum_ia_ix d(b_i)=0
$$
for all $x\in RC$. In view of Proposition \ref{pro3} (i), we get $\sum_ia_iyb_i=0$ for all $y\in RC$. So $\phi=0$.
\end{proof}

As an application of Theorem \ref{thm32}, we get the following.

\begin{lem}\label{lem17}
Let $R$ be a prime ring, and $a_{jk}\in RC$, $1\leq j\leq m$, $1\leq k\leq n_j$. Then
$$
\sum_{j=1}^m\big[a_{j1},\ldots,a_{jn_j}, [x, y]\big]=0\ \forall x, y\in RC\ \text{\rm iff}\ \ \sum_{j=1}^m(-1)^{n_j}\big[a_{jn_j},\ldots,a_{j1}, z\big]\in C\ \forall z\in RC.
$$
\end{lem}

\begin{proof}
Let $\phi:=\sum_{j=1}^m\ad_{a_{j1}}\cdots\ad_{a_{jn_j}}$. Then $\phi^*=\sum_{j=1}^m(-1)^{n_j}\ad_{a_{jn_j}}\cdots\ad_{a_{j1}}$.
We have
\begin{eqnarray*}
\sum_{j=1}^m\big[a_{j1},\ldots,a_{jn_j}, [x, y]\big]=0\ \forall x, y\in RC&\Leftrightarrow& \phi([RC, RC])=0\\
                                                                     &\Leftrightarrow&  \phi^*(RC)\subseteq C\ (\text{\rm by Theorem \ref{thm32}})\\
                                                                     &\Leftrightarrow&\sum_{j=1}^m(-1)^{n_j}\ad_{a_{jn_j}}\cdots\ad_{a_{j1}}(RC)\subseteq C\\
                                                                     &\Leftrightarrow&\sum_{j=1}^m(-1)^{n_j}\big[a_{jn_j},\ldots,a_{j1}, z\big]\in C\ \forall z\in RC,
 \end{eqnarray*}
as desired.
\end{proof}

We next turn to the following theorem.

\begin{thm}\label{thm25}
Let $R$ be an exceptional prime ring, $a_1,\ldots,a_n\in R$, and let $I$ be a nonzero ideal of $R$.
Then $\big[a_1, a_2,\ldots,a_n, [I, I]\big]\subseteq Z(R)$ iff one of $a_1,\ldots,a_n$ lies in $[RC, RC]$.
\end{thm}

\begin{proof}
Clearly, if one of $a_1,\ldots,a_n$ lies in $[RC, RC]$, it follows from Lemma \ref{lem20} that
$$
\big[a_1, a_2,\ldots,a_n, [I, I]\big]=\big[a_1, [a_2, [\cdots, [a_n, [I, I]]\cdots]]\big]\subseteq Z(R).
$$
Conversely, assume that
$
\big[a_1, a_2,\ldots,a_n, [I, I]\big]\subseteq Z(R).
$
Note that $IC=RC$. We get
$
\big[a_1, a_2,\ldots,a_n, [RC, RC]\big]\subseteq C.
$
We proceed the proof by induction $n$.
Suppose that $a_1\notin [RC, RC]$.
Then
$RC=[RC, RC]+Ca_1$.
By Lemma \ref{lem20}
$$
\big[[RC, RC], \big[a_2,\ldots,a_n, [RC, RC]\big]\big]\subseteq C,
$$
this implies that
 \begin{eqnarray*}
&&\big[RC, \big[a_2,\cdots,a_n, [RC, RC]\big]\big]\\
&=&\big[Ca_1+[RC, RC], \big[a_2,\ldots,a_n, [RC, RC]\big]\big]\\
&\subseteq&\big[Ca_1, \big[a_2,\ldots,a_n, [RC, RC]\big]\big]+\big[[RC, RC], \big[a_2,\ldots,a_n, [RC, RC]\big]\big]\\
&\subseteq &C.
\end{eqnarray*}
By Lemma \ref{lem8} (ii), we get
$
\big[a_2,\ldots,a_n, [RC, RC]\big]\subseteq C.
$
By induction on $n$, one of $a_2,\ldots,a_n$ lies in $[RC, RC]$.
\end{proof}

\begin{thm}\label{thm28}
Let $R$ be an exceptional prime ring, $a_1,\ldots,a_n\in R$ with $a_n\notin Z(R)$, $n\geq 2$ and let $I$ be a nonzero ideal of $R$.
Then $\big[a_1, a_2,\cdots,a_n, I\big]\subseteq Z(R)$ iff one of $a_1,\ldots,a_{n-1}$ lies in $[RC, RC]$.
\end{thm}

\begin{proof}
By Lemma \ref{lem20}, the part ``$\Leftarrow$" is clear.

``$\Rightarrow$"\ \ Assume that $\big[a_1, a_2,\cdots,a_n, I\big]\subseteq Z(R)$.
Suppose that none of $a_1,\ldots,a_{n-1}$ lies in $[RC, RC]$.
Then $RC=Ca_1+[RC, RC]$. Since $n\geq 2$, by Lemma \ref{lem20}  we get
$$
\big[[RC, RC],\big[a_2,\ldots,a_n, I\big]\big]\subseteq C
$$
and so
 \begin{eqnarray*}
&&\big[RC, \big[a_2,\ldots,a_n, I\big]\big]\\
&\subseteq&\big[Ca_1, \big[a_2,\ldots,a_n, I\big]\big]+\big[[RC, RC], \big[a_2,\ldots,a_n, I\big]\big]\subseteq C.\\
\end{eqnarray*}
In view of Lemma \ref{lem8} (ii), $\big[a_2,\cdots,a_n, I\big]\subseteq C$. By induction, we get
$[a_n, I]\subseteq C$ and so $a_n\in Z(R)$, a contradiction.
\end{proof}

\begin{thm}\label{thm31}
Let $R$ be an exceptional prime ring, $a_1,\ldots,a_n\in R$ with $a_1\notin Z(R)$, $n\geq 2$ and let $I$ be a nonzero ideal of $R$.
Then $\big[a_1, a_2,\ldots,a_n, [I, I]\big]=0$ iff one of $a_2,\ldots,a_{n}$ lies in $[RC, RC]$.
\end{thm}

\begin{proof}
Note that $IC=RC$. We have
\begin{eqnarray*}
\big[a_1, a_2,\ldots,a_n, [I, I]\big]=0 &\Leftrightarrow&\big[a_1, a_2,\ldots,a_n, [RC, RC]\big]=0\\
                                                                   &\Leftrightarrow&\big[a_n, a_{n-1},\ldots,a_1, RC\big]\subseteq C\ (\text{\rm by\ Lemma \ref{lem17}})\\
                                                                   &\Leftrightarrow& a_i\in [RC, RC]\ \text{\rm for some}\ i\geq 2\  (\text{\rm by\ Theorem \ref{thm28}}),
\end{eqnarray*}
as desired.
\end{proof}

Finally, we prove Theorem C in the introduction.

\begin{thm}\label{thm36}
Let $R$ be an exceptional prime ring, and let $L$ be a noncentral Lie ideal of $R$, and $a_i\in R$, $1\leq i\leq n$.

(i)\ If $L$ is either abelian or nonabelian of Type I, then $\big[a_1, a_2,\cdots,a_n, L\big]\subseteq Z(R)$ iff one of $a_1,\ldots,a_n$ lies in $[RC, RC]$.

(ii)\ If $L$ is nonabelian of Type II, then $\big[a_1, a_2,\cdots,a_n, L\big]\subseteq Z(R)$ iff either $a_n\in Z(R)$ or one of $a_1,\ldots,a_{n-1}$ lies in $[RC, RC]$.
\end{thm}

\begin{proof}

(i)\ Assume that $L$ is nonabelian of Type I. Thus $LC=[RC, RC]$ and so
\begin{eqnarray*}
\big[a_1, a_2,\cdots,a_n, L\big]\subseteq Z(R)&\Leftrightarrow& \big[a_1, a_2,\cdots,a_n, [RC, RC]\big]\subseteq C\\
                                                                                     &\Leftrightarrow& \text{\rm one of}\ a_1,\ldots,a_n\in [RC, RC]\ \ (\text{\rm by Theorem \ref{thm25}}).
\end{eqnarray*}
Assume next that $L$ is abelian It follows from Lemma \ref{lem19} (i)
that $LC=[w, RC]$ for some $w\in L\setminus Z(R)$. Thus
\begin{eqnarray*}
\big[a_1, a_2,\cdots,a_n, L\big]\subseteq Z(R)&\Leftrightarrow&  \big[a_1, a_2,\cdots,a_n, [w, RC]\big]\subseteq C\\
                                                                                     &\Leftrightarrow&  a_i\in [RC, RC]\ \text{\rm for some}\ i\
                                                                                     (\text{\rm by Theorem \ref{thm28}}).
\end{eqnarray*}

(ii)\ Assume that $L$ is nonabelian of Type II. Thus $LC=RC$ and so
\begin{eqnarray*}
\big[a_1, a_2,\cdots,a_n, L\big]\subseteq Z(R)&\Leftrightarrow& \big[a_1, a_2,\cdots,a_n, RC\big]\subseteq C\\
                                                                                     &\Leftrightarrow& \text{\rm either}\ a_n\in Z(R)\ \text{\rm or}\  \exists i, 1\leq i\leq n-1, a_i\in [RC, RC]\\\
                                                                                     && (\text{\rm by Theorem \ref{thm28}}).
\end{eqnarray*}
\end{proof}

For the nilpotency of derivations (in particular, inner derivations) of arbitrary prime rings, we refer the reader to \cite{martindale1983, chung1983, chung1984, chung1984-1, lee1986, chuang2005}.

\section{\bf Theorem D}
Let $x_1,\ldots,x_m, y\in R$, where $m>1$. Clearly, we have
$$
E_m(x_1,\ldots,x_m)=(-1)^{m-1}[x_m, x_{m-1},\ldots,x_1]
$$
and
$$
\big[E_m(x_1,\ldots,x_m), y\big]=\sum_{i=1}^mE_m(x_1,\ldots,[x_i, y],\ldots,x_m).
$$
Thus if $L$ is a Lie ideal of $R$ then so is $E_m(L)^+$. Moreover, $E_m(L)^+=\big[L, E_{m-1}(L)^+\big]$ for $m>1$.

\begin{lem}\label{lem22}
Let $R$ be an exceptional prime ring.

(i)\ $E_m(RC)^+=[RC, RC]$ for $m\geq 2$;

(ii)\ $E_2([RC, RC])^+=C$ and $E_m([RC, RC])^+=0$ for $m\geq 3$.
\end{lem}

\begin{proof}
(i)\ Clearly, $E_2(RC)^+=[RC, RC]$. By induction on $m$, for $m>2$, we have
\begin{eqnarray*}
E_m(RC)^+&=&\big[RC, E_{m-1}(RC)^+\big]\\
                     &=&\big[RC, [RC, RC]\big]\\
                     &=&\big[RC, \overline{[RC, RC]}\big]\\
                     &=&[RC, RC]\ (\text{\rm by Theorem \ref{thm24}}),
\end{eqnarray*}
as desired.

(ii)\ It follows from Lemma \ref{lem20} that $E_2([RC, RC])^+=C$. Let $m\geq 3$. By induction on $m$, $E_{m-1}([RC, RC])^+=0$ if $m>3$ and $E_{m-1}([RC, RC])^+=C$ if $m=3$. Thus
$
E_m(RC)^+=\big[RC, E_{m-1}(RC)^+\big]=0.
$
\end{proof}

The following is a consequence of \cite[Theorem 4.5]{lee2025}.

\begin{lem}\label{lem23}
Let $R$ be an exceptional prime ring. Then a noncentral Lie ideal $L$ of $R$ is contained in $[RC, RC]$ iff
$L$ is either abelian or nonabelian of Type I.
\end{lem}

\begin{lem}\label{lem21}
Let $R$ be a prime ring, and let $L$ be a noncentral Lie ideal of $R$. Then $E_m(L)^+\subseteq Z(R)$ iff $R$ is exceptional, $m>1$, and $L\subseteq [RC, RC]$.
\end{lem}

\begin{proof}
``$\Rightarrow$":\ Assume that $E_m(L)^+\subseteq Z(R)$. That is, $\big[E_{m-1}(L)^+, L\big]\subseteq Z(R)$.
If $R$ is nonexceptional, then, by Lemma \ref{lem8} (v), $E_{m-1}(L)^+\subseteq Z(R)$ and finally we get $L\subseteq Z(R)$, a contradiction.
Thus $R$ is exceptional. Clearly, $m>1$. Let $x\in L$ and let $x_j:=x$ for $j\geq 2$. We have
\begin{eqnarray}
E_m(L, x_2,\ldots,x_m) &=&[x_m, x_{m-1},\ldots,x_2, L]\subseteq Z(R).
\label{eq:124}
\end{eqnarray}

Case 1:\ $L$ is either abelian or nonabelian of Type I. It follows from Lemma \ref{lem23} that $L\subseteq [RC, RC]$, as desired.

Case 2:\  $L$ is nonabelian of Type II. In view of Theorem \ref{thm36} (ii), either $x\in C$ or $x\in [RC, RC]$. Since $C\subseteq [RC, RC]$, we get $x\in [RC, RC]$.
That is, $L\subseteq [RC, RC]$, as desired.

``$\Leftarrow$":\ It follows directly from Lemma \ref{lem22} (ii).
\end{proof}

We are now ready to prove Theorem D in the introduction.

\begin{thm}\label{thm37}
Let $R$ be a prime ring, and let $L$ be a noncentral Lie ideal of $R$, and either $m>1$ or $ k>1$. Then $\big[E_m(L)^+, E_k(L)^+\big]=0$ iff $R$ is exceptional and $L\subseteq [RC, RC]$.
\end{thm}

\begin{proof}
``$\Rightarrow$":\ Assume that $\big[E_m(L)^+, E_k(L)^+\big]=0$, that is, Eq.\eqref{eq:123} holds
for all $x_i, y_j\in L$.
We claim that $R$ is exceptional. Indeed , if $L$ is abelian, it follows from \cite[Theorem 6.3]{lee2025} that $R$ is exceptional.
Suppose next that $L$ is nonabelian. Then Eq.\eqref{eq:123} implies that $R$ is exceptional (see \cite[Theorem 2]{lanski2010}).

Note that $E_m(LC)^+$ is a Lie ideal of $RC$.
By Eq.\eqref{eq:123}, $\big[E_m(LC)^+, E_k(LC)^+\big]=0$.
In view of Theorem \ref{thm14}, $E_m(LC)^+\subseteq C$, $E_m(LC)^+=Cw+C$ for some $w\in E_m(L)^+\setminus Z(R)$, or $[RC, RC]\subseteq E_m(LC)^+$.

Case 1:\ $E_m(LC)^+\subseteq C$. By Lemma \ref{lem21}, $LC\subseteq [RC, RC]$.

Case 2:\ $E_m(LC)^+=Cw+C$ for some $w\in E_m(L)^+\setminus Z(R)$. In particular, $Cw+C\subseteq LC$. In view of Theorem \ref{thm14}, $LC$ is equal to $Cw+C$, $[RC, RC]$, or $RC$. If $LC=RC$,  it follows from Lemma \ref{lem22} (i) that $E_m(LC)^+=[RC, RC]$, a contradiction. Hence $L$ is either abelian or nonabelian of Type I. Hence $LC\subseteq [RC, RC]$.

Case 3:\ $[RC, RC]\subseteq E_m(LC)^+$. Then $[RC, RC]\subseteq LC$, implying that either $LC=[RC, RC]$ or $RC$. Suppose on the contrary that
$LC=RC$. In view of Lemma \ref{lem22}, $E_s(RC)^+=[RC, RC]$ for $s\geq 2$. Clearly, $E_1(RC)^+=RC$.
Since $m>1$ or $ k>1$, we get
$$
C=\big[[RC, RC], [RC, RC]\big]\subseteq \big[E_m(L)^+, E_k(L)^+\big]=0,
$$
a contradiction.

``$\Leftarrow$":\ Case 1: $L$ is abelian. Then, by Lemma \ref{lem22} (iii), $E_s(L)^+=0$ for $s\geq 2$. Thus Eq.\eqref{eq:123} holds
for all $x_i, y_j\in L$.

Case 2:\ $L$ is nonabelian of Type I. Then $LC=[RC, RC]$. By Lemma \ref{lem22} (ii), $E_2([RC, RC])^+=C$ and $E_s([RC, RC])^+=0$ for $s\geq 3$.
Since $m>1$ or $k>1$, we get  $\big[E_m(LC)^+, E_k(LC)^+\big]=0$, as desired.
\end{proof}

\section{Problem \ref{problem1}: The abelian case}
In this section, {\it $R$ is always an exceptional prime ring with extended centroid $C$.}

\begin{lem}\label{lem14}
A derivation $d$ of $R$ is of the form
$\text{\rm ad}_g$ for some $g\in [RC, RC]$ iff $d([R, R])\subseteq Z(R)$.
\end{lem}

\begin{proof}
By Lemma \ref{lem20}, the part ``$\Rightarrow$" is clear. For the part ``$\Leftarrow$", assume that $d([R, R])\subseteq Z(R)$.
Let $\beta\in Z(R)$ and $x\in [R, R]$.
Then $\beta x\in [R, R]$ and so
$
d(\beta x)=\beta d(x)+d(\beta)x\in Z(R),
$
implying that $d(\beta)x\in C$. That is, $d(\beta)[R, R]\subseteq Z(R)$. Since $R$ is not commutative, we get $d(\beta)=0$.
That is, $d(Z(R))=0$. By the Skolem-Noether theorem, $d$ is X-inner, i.e., $d=\text{\rm ad}_g$ for some $g\in RC$. Then
$\big[g, [RC, RC]\big]\subseteq C$. In view of Theorem \ref{thm25}, $g\in [RC, RC]$.
\end{proof}

\begin{lem}\label{lem15}
Let $\delta, d$ be X-outer derivations of $R$. Then there exists $\beta\in Z(R)$ such that neither $\delta(\beta)=0$ nor $d(\beta)=0$.
\end{lem}

\begin{proof}
Suppose not, that is, given any $\beta\in Z(R)$, either $\delta(\beta)=0$ or $d(\beta)=0$. Then
$Z(R)$ is the union of its two additive subgroups:
$$
\{\beta\in Z(R)\mid \delta(\beta)=0\}\ \text{\rm and}\ \ \{\beta\in Z(R)\mid d(\beta)=0\}.
$$
Since $Z(R)$ cannot be the union of its two proper subgroups,
this implies that either $\delta(Z(R))=0$ or $d(Z(R))=0$. By the Skolem-Noether theorem, one of $\delta$ and $d$ is X-inner, a contradiction.
\end{proof}

The following theorem solves Problem \ref{problem1} when $L$ is noncentral abelian.

\begin{thm}\label{thm34}
Let $\delta, d$ be nonzero derivations of $R$, and let $L$ be a noncentral abelian Lie ideal of $R$.

(i)\ If $d$ is X-inner and $\delta$ is X-outer, then $\delta d(L)\subseteq Z(R)$ iff $d([R, R])\subseteq Z(R)$;

(ii)\ If both $d$ and $\delta$ are X-inner, then $\delta d(L)\subseteq Z(R)$ iff either $d([R, R])\subseteq Z(R)$ or $\delta([R, R])\subseteq Z(R)$;

(iii)\ If $d$ is X-outer and $\delta$ is X-inner, then $\delta d(L)\subseteq Z(R)$ iff $\delta([R, R])\subseteq Z(R)$;

(iv)\ If both $\delta$ and $d$ are X-outer, then $\delta d(L)\subseteq Z(R)$ iff there exists $\beta\in Z(R)$ such that neither $d(\beta)=0$ nor $\delta(\beta)=0$, and there exist $g\in RC$, $\mu\in C$, and $h\in [RC, RC]$ such that
\begin{eqnarray}
\delta(\beta)d+d(\beta)\delta=\text{\rm ad}_g,\ d^2=\mu d+\text{\rm ad}_h\ \text{\rm and}\ \ \delta(\beta)\mu z+[g, z]\in C\ \forall z\in LC,
\label{eq:121}
\end{eqnarray}
and, moreover, if neither $d(LC)\subseteq LC$ nor $g\in [RC, RC]$, then $g^2+\delta(\beta)\mu g\in LC$.
\end{thm}

\begin{proof}
Since $L$ is a noncentral abelian Lie ideal of $R$, it follows from Remark \ref{remark1} (i) that $L$ and $LC$ satisfy the same differential identities. Thus $\delta d(L)\subseteq Z(R)$ iff $\delta d(LC)\subseteq C$, and $d(L)\subseteq Z(R)$
iff $d(LC)\subseteq C$.

(i)\ Assume that $d$ is X-inner and $\delta$ is X-outer.
We can write $d=\text{\rm ad}_a$ for some $a\in RC$.

``$\Rightarrow$":\  Let $z\in LC$ and $x\in RC$. Since $\delta d(L)\subseteq Z(R)$, we get
$
\delta(\big[a, [z, x]\big])\in C
$
and hence
$$
\big[\delta(a), [z, x]\big]+\big[a, [\delta(z), x]\big]+\big[a, [z, \delta(x)]\big]\in C.
$$
In view of Proposition \ref{pro3}, $\big[a, [z, RC]\big]\subseteq C$ for all $z\in LC$. By Theorem \ref{thm28}, we get
$a\in [RC, RC]$ and so $d([R, R])\subseteq Z(R)$, as desired.

``$\Leftarrow$": Since $d([R, R])\subseteq Z(R)$, it follows from Lemma \ref{lem14} that $a\in [RC, RC]$. Note that
$LC\subseteq [RC, RC]$. Thus
$$
\delta d(L)=\delta \big([a, L]\big)\subseteq \delta \big(\big[[RC, RC], [RC, RC]\big]\big)\subseteq \delta(C)\subseteq C.
$$
Hence $\delta d(L)\subseteq Z(R)$.

(ii)\ Write $d=\text{\rm ad}_a$ and $\delta=\text{\rm ad}_b$ for some $a, b\in RC$

``$\Rightarrow$":\  Since $\delta d(L)\subseteq Z(R)$, we get
$
\big[b, \big[a, [z, RC]\big]\big]\subseteq C
$
for all $z\in LC$. Note that $LC\nsubseteq C$.  It follows from  Theorem \ref{thm28} that either $a\in [RC, RC]$ or $b\in [RC, RC]$.
That is, either $d([R, R])\subseteq Z(R)$ or $\delta([R, R])\subseteq Z(R)$ (see Lemma \ref{lem14}).

The proof of ``$\Leftarrow$" is similar to that of ``$\Leftarrow$" of (i).

(iii)\ Assume that $d$ is X-outer and $\delta$ is X-inner. Write $\delta=\text{\rm ad}_b$ for some $b\in RC$.

``$\Rightarrow$":\ Let $z\in LC$ and $x\in RC$. Then $[z, x]\in LC$. Since $\delta d(L)\subseteq Z(R)$, we get
$
\big[b, d([z, x])\big]\in C
$
and so
$
\big[b, [d(z), x]\big]+\big[b,[z, d(x)]\big]\in C.
$
It follows from Proposition \ref{pro3} that $\big[b, [z, RC]\big]\in C$ for all $z\in LC$ and $x\in RC$. By Theorem \ref{thm28}, $b\in [RC, RC]$
and so $\delta([R, R])\subseteq Z(R)$ (see Lemma \ref{lem14}).

The proof of ``$\Leftarrow$" is similar to that of ``$\Leftarrow$" of (i).

(iv)\ Assume that both $\delta$ and $d$ are X-outer.
In view of Lemma \ref{lem15}, there exists $\beta\in Z(R)$ such that neither $d(\beta)=0$ nor $\delta(\beta)=0$.

``$\Rightarrow$":\ Let $z\in LC$ and $\beta\in C$. Then
$$
\delta d(\beta z)=\delta d(\beta)z + \big(\delta(\beta)d+d(\beta)\delta\big)(z)+\beta\delta d(z)\in C,
$$
and so
 \begin{eqnarray}
\eta(z)+\delta d(\beta)z\in C,
\label{eq:2}
\end{eqnarray}
where $\eta:=\delta(\beta)d+d(\beta)\delta$.
Suppose on the contrary that $\eta$ is X-outer. By the Skolem-Noether theorem, $\eta(\alpha)\ne 0$ for some $\alpha\in C$.
Let $z\in LC$. Then, by Eq.\eqref{eq:2}, $\eta(\alpha z)+\delta d(\beta)\alpha z\in C$, implying that $\eta(\alpha)z\in C$. Thus $LC\subseteq C$, a contradiction. Thus $\eta$ is X-inner.
 Write $\eta=\delta(\beta)d+d(\beta)\delta=\text{\rm ad}_g$ for some $g\in RC$.

Let $z\in LC$ and $y\in RC$. Then $[z, y]\in LC$ and so
\begin{eqnarray*}
\delta(\beta)d^2([z, y])+[g, d([z, y])]=d(\beta)\delta d([z, y])\in C.
\end{eqnarray*}
Note $d^2$ is a derivation of $RC$. Thus
\begin{eqnarray}
\delta(\beta)\big([d^2(z), y]+[z, d^2(y)]\big)+\big[g,[d(z), y]+[z, d(y)]\big]\in C.
\label{eq:113}
\end{eqnarray}
Suppose on the contrary that $d^2$ and $d$ are $C$-independent modulo X-inner derivations.
Applying Proposition \ref{pro3} (ii) to Eq.\eqref{eq:113}, we get
$$
\delta(\beta)\big([d^2(z), y]+[z, y_2]\big)+\big[g,[d(z), y]+[z, y_1]\big]\in C
$$
for all $z\in LC$ and $y, y_1, y_2\in RC$.
In particular,
$\delta(\beta)[z, y_2]\in C$ for all $z\in LC$ and $y_2\in RC$, that is,
$\delta(\beta)\big[RC, [LC, RC]\big]=0$, implying that $LC\subseteq C$ (see Lemma \ref{lem8} (ii)), a contradiction.

Thus $d^2=\mu d+\text{\rm ad}_h$ for some $\mu\in C$ and $h\in RC$.
  Let
 $$
 \varphi(z):=\delta(\beta)d^2(z)+[g, d(z)]=d(\beta)\delta d(z)\in C
 $$
 for $z\in LC$. Note that $d^2$ is a derivation of $RC$. Moreover, $d^2(\alpha)=\mu d(\alpha)$ for $\alpha\in C$.
Let $z\in LC$. Since $d$ is X-outer, $d(\alpha)\ne 0$ for some $\alpha\in C$.
Then
\begin{eqnarray*}
\varphi(\alpha z)+\alpha\varphi(z)&=&\delta(\beta)d^2(\alpha z)+[g, d(\alpha z)]+\alpha\delta(\beta)d^2(z)+\alpha[g, d(z)]\\
                                                               &=&\delta(\beta)d^2(\alpha)z+d(\alpha)[g, z]\\
                                                               &=&d(\alpha)\big(\delta(\beta)\mu z+[g, z]\big)\in C.
\end{eqnarray*}
Hence
\begin{eqnarray}
\delta(\beta)\mu z+[g, z]\in C
\label{eq:117}
\end{eqnarray}
for all $z\in LC$.  By Eq.\eqref{eq:113} we get
\begin{eqnarray}
\delta(\beta)\big([d^2(z), y]+\big[z, \mu\,d(y)+[h, y]\big]\big)+\big[g,[d(z), y]+[z, d(y)]\big]\in C
\label{eq:115}
\end{eqnarray}
for all $z\in LC$ and $y\in RC$. Applying Proposition \ref{pro3} to Eq.\eqref{eq:115}, we get
\begin{eqnarray*}
\delta(\beta)\big([d^2(z), y]+\big[z, \mu\,y_1+[h, y]\big]\big)+\big[g, [d(z), y]+[z, y_1]\big]\in C
\end{eqnarray*}
for all $z\in LC$ and $y, y_1\in RC$.  Clearly, $\big[z, [h, y]\big]\in C$.
This implies that
\begin{eqnarray*}
\delta(\beta)[d^2(z), y]+\big[g, [d(z), y]\big]\in C
\end{eqnarray*}
for all $z\in LC$ and $y\in RC$.
Thus
\begin{eqnarray}
\delta(\beta)[d^2(z), y]+\big[g, [d(z), y]\big]&=&\delta(\beta)\big[\mu d(z)+[h, z], y\big]+\big[g, [d(z), y]\big]\nonumber\\
                                                                                   &=&\delta(\beta)\mu[d(z), y]+\big[g, [d(z), y]\big]+\delta(\beta)\big[[h, z], y\big]\nonumber\\
                                                                                   &&\in C
                                                                                   \label{eq:116}
\end{eqnarray}
for all $z\in LC$ and $y\in RC$.

{\bf Case 1}:\ $d(LC)\subseteq LC$. Then $[d(z), y]\in LC$ for $z\in LC$ and $y\in RC$. In view of Eq.\eqref{eq:117}, $\delta(\beta)\mu[d(z), y]+\big[g, [d(z), y]\big]\in C$. Thus, by Eq.\eqref{eq:116}, we have $\big[[h, LC], RC\big]\subseteq C$ and so $[h, LC]\subseteq C$ (see Lemma \ref{lem8} (ii)).
If $[h, LC]\ne 0$, it follows from \cite[Lemma 2.5]{lee2025} that $h\in [RC, RC]$. Suppose that $[h, LC]=0$. It follows from \cite[Lemma 5.1]{lee2025} that $h\in LC\subseteq [RC, RC]$.

{\bf Case 2}:\ $d(LC)\nsubseteq LC$. Note that $d(LC)+LC$ is a Lie ideal of $RC$ and is a $C$-subspace of $RC$, implying that
$[RC, RC]=d(LC)+LC$.
In view of Eq.\eqref{eq:116} and Eq.\eqref{eq:117}, we have
\begin{eqnarray*}
&&\delta(\beta)\mu[d(z)+w, y]+\big[g, [d(z)+w, y]\big]+\delta(\beta)\big[[h, z], y\big]\\
&=&\big(\delta(\beta)\mu[d(z), y]+\big[g, [d(z), y]\big]+\delta(\beta)\big[[h, z], y\big]\big)\\
&&\ \ \ \ \ \ \ \ \ \ \       + \big(\delta(\beta)\mu[w, y] +\big[g, [w, y]\big]\big)\in C
\end{eqnarray*}
for all $z, w\in LC$ and $y\in RC$. Since  $[RC, RC]=d(LC)+LC$, we get
\begin{eqnarray}
\delta(\beta)\mu[[x_1, x_2], y]+\big[g, [[x_1, x_2], y]\big]\in \delta(\beta)\big[[h, LC], RC\big]+C\subseteq LC
       \label{eq:122}
\end{eqnarray}
for all $x_1, x_2, y\in RC$. By Theorem \ref{thm24}, we have
$$
\big[[RC, RC], RC\big]=\big[\overline{[RC, RC]}, RC\big]=[RC, RC].
$$
By Eq.\eqref{eq:122}, we get
\begin{eqnarray}
\delta(\beta)\mu x+[g, x]\in LC
       \label{eq:119}
\end{eqnarray}
for all $x\in [RC, RC]$.

{\bf Subcase 1}:\ $g\in [RC, RC]$. Then $[g, x]\in C\subseteq  LC$ for all $x\in [RC, RC]$.
By Eq.\eqref{eq:119}, we have $\delta(\beta)\mu [RC, RC]\subseteq LC$, implying that $\mu=0$. Thus $d^2=\ad_h$.
We claim that $h\in [RC, RC]$. Indeed, let $z\in LC$. Then
$$
d(\beta)\delta d(z)=\delta(\beta)d^2(z)+[g, d(z)]=\delta(\beta)[h, z]+[g, d(z)]\in C.
$$
Since $[g, d(z)]\in \big[[RC, RC], [RC, RC]\big]\subseteq C$, we get $\delta(\beta)[h, z]\in C$. Hence $[h, LC]\subseteq C$.
If $[h, LC]\ne 0$, it follows from \cite[Lemma 4.5]{lee2025} that $h\in [RC, RC]$. Otherwise, $[h, LC]=0$. Then $\big[h, [LC, RC]\big]=0$.
In view of Lemma \ref{lem2}, $h\in LC\subseteq [RC, RC]$, as desired.

{\bf Subcase 2}:\ $g\notin [RC, RC]$. We claim that $g^2+\delta(\beta)\mu g\in LC$ and $h\in [RC, RC]$.
In this case, $RC=Cg+[RC, RC]$.
Let $y\in RC$. Write $y=x+\alpha g$ for some $x\in [RC, RC]$ and $\alpha \in C$.
Thus
$$
\delta(\beta)\mu y+[g, y]=\big(\delta(\beta)\mu x+[g, x]\big)+\delta(\beta)\mu\alpha g\in LC+Cg.
$$
Commuting it with $g$, we get
$$
\big[g^2+\delta(\beta)\mu g, y\big]=\big[g, \delta(\beta)\mu y+[g, y]\big]\in [g, LC+Cg]\subseteq LC
$$
for all $y\in RC$. This implies that
$
\big[LC, \big[g^2+\delta(\beta)\mu g, RC\big]\big]=0
$
and so
$$
\big[g^2+\delta(\beta)\mu g, \big[LC, RC\big]\big]=0,
$$
since $(\ad_a\ad_b)^*=\ad_b\ad_a$ for $a, b\in RC$.
In view of Lemma \ref{lem2}, we get $g^2+\delta(\beta)\mu g\in LC$, as desired.
The final step is to prove that $h\in [RC, RC]$.
Let $z\in LC$. Then $d(z)\in [RC, RC]$ and so
\begin{eqnarray*}
d(\beta)\delta d(z)&=&\delta(\beta)d^2(z)+[g, d(z)]\\
                                  &=&\delta(\beta)\big(\mu d(z)+[h, z]\big)+[g, d(z)]\\
                                   &=&\big(\delta(\beta)\mu d(z)+[g, d(z)]\big)+\delta(\beta)[h, z]\in C.
\end{eqnarray*}
Then
\begin{eqnarray*}
0&=&\big[g, d(\beta)\delta d(z)\big]\\
  &=&\big[g, \big(\delta(\beta)\mu d(z)+[g, d(z)]\big)+\delta(\beta)[h, z]\big]\\
  &=&\big[g^2+\delta(\beta)\mu g, d(z)\big]+\delta(\beta)\big[g, [h, z]\big].
\end{eqnarray*}
Since $\big[g^2+\delta(\beta)\mu g, d(z)\big]\in [LC, d(LC)]\subseteq \big[[RC, RC], [RC, RC]\big]\subseteq C$, we get
$\big[g, [h, z]\big]\in C$. That is,
$
\big[g, [h, LC]\big]\subseteq  C
$
and so $\big[g, [h, [LC, RC]]\big]\subseteq  C$. Recall that $LC\nsubseteq C$. In view of Theorem \ref{thm28}, $h\in [RC, RC]$, as claimed.

``$\Leftarrow$": For $a, b\in RC$, we let $a\equiv b$ stand for $a-b\in C$. Let $z\in LC$.
Then $d(z)\in [RC, RC]$ and $[h, z]\in C$ as $h\in [RC, RC]$. In view of Eq.\eqref{eq:121}, we have
\begin{eqnarray}
d(\beta)\delta d(z)&=&\big(\delta(\beta)\mu d(z)+[g, d(z)]\big)+\delta(\beta)[h, z]\nonumber\\
                                   &\equiv&\delta(\beta)\mu d(z)+[g, d(z)].
\label{eq:120}
\end{eqnarray}
for all $z\in LC$. If $d(LC)\subseteq LC$, then $d(z)\in LC$ and it follows from Eq.\eqref{eq:121} that $\delta(\beta)\big(\mu d(z)+[g, d(z)]\big)\in C$, that is, $\delta d(z)\in C$, as desired. Suppose that $g\in [RC, RC]$. Then $[g, z]\in C$. Since $\delta(\beta)\mu z+[g, z]\in C$ (see Eq.\eqref{eq:121}), we get $\delta(\beta)\mu LC\subseteq C$ and hence $\mu=0$. By Eq.\eqref{eq:120}, it follows that $\delta d(z)\in C$, as desired.

Thus we may assume next that neither $d(LC)\subseteq LC$ nor $g\in [RC, RC]$. Then, by assumption, $g^2+\delta(\beta)\mu g\in LC$. By Eq.\eqref{eq:120}, we have
\begin{eqnarray*}
\big[g, d(\beta)\delta d(z)\big]&\equiv& \big[g, \delta(\beta)\mu d(z)+[g, d(z)]\big]\\
                                                         &\equiv& \big[g^2+\delta(\beta)\mu g, d(z)\big]\\
                                                         &\in& \big[LC, d(LC)\big]\subseteq \big[[RC, RC], [RC, RC]\big]\subseteq C
\end{eqnarray*}
for all $z\in LC$. Since $g\notin [RC, RC]$, it follows that $RC=Cg+[RC, RC]$. Thus, let $z\in LC$, we have
\begin{eqnarray*}
\big[RC, d(\beta)\delta d(z)\big]&=&\big[Cg+[RC, RC], d(\beta)\delta d(z)\big]\\
                                                            &\subseteq & \big[Cg, d(\beta)\delta d(z)\big]+\big[[RC, RC], d(\beta)\delta d(z)\big]\\
                                                             &\subseteq &C,
\end{eqnarray*}
since $\big[[RC, RC], d(\beta)\delta d(z)\big]\subseteq \big[[RC, RC], [RC, RC]\big]\subseteq C$.
Hence $d(\beta)\delta d(z)\in C$, completing the proof.
\end{proof}

\section{Problem \ref{problem1}: The nonabelian case}
In this section, {\it $R$ is always an exceptional prime ring with extended centroid $C$.}

Let $K$ be a noncentral abelian Lie ideal of $RC$. We claim that $KC=K$. Indeed, it follows from \cite[Lemma 4.1]{lee2025} that $KC=[a, RC]$ for some $a\in K\setminus C$. Since $K$ is a Lie ideal of $RC$, we get $KC=[a, RC]\subseteq K$ and so $K=KC$.

\begin{lem}\label{lem18}
Let  $d$ be an X-outer derivation of $R$. Then there exist infinitely many distinct noncentral abelian Lie ideals $L_i$ of $RC$ satisfying $d(L_i)\nsubseteq L_i$ for all $i$.
\end{lem}

\begin{proof}
Note that every noncentral abelian Lie ideal of $RC$ is of the form $Cw+C$ for some $w\in [RC, RC]\setminus C$ (see \cite[Theorem 4.4]{lee2025}).
Moreover, two noncentral abelian Lie ideals $K, L$ of $RC$ are distinct iff $[K, L]\ne 0$ (see \cite[Theorem 6.3]{lee2025}).
We also notice that a noncentral abelian Lie ideal $K=Cw+C$, where $w\in [RC, RC]\setminus C$, satisfies $d(K)\nsubseteq K$ iff $d(w^2)\ne 0$. Indeed, if $d(K)\subseteq K$, then $d(w)\in Cw+C$ and so $[d(w), w]=0$ and so $d(w^2)=0$. Conversely, if $[d(w), w]=0$, then, by \cite[Lemma 5.1]{lee2025}, $d(w)\in Cw+C$ and so $d(K)\subseteq K$, as desired.

Since $d$ is an X-outer derivation of $R$, it follows from the Skolem-Noether theorem that $d(C)\ne 0$. Note that $C=\big[[RC, RC], [RC, RC]\big]$ and so
$
d(\big[[RC, RC], [RC, RC]\big])\ne 0.
$
Then there exists $u\in  [RC, RC]$ such that $d(u^2)\ne 0$. Otherwise, $d(u^2)=0$ for all $u\in [RC, RC]$. Linearizing it, we get $d([u, v])=0$ for all
$u, v\in [RC, RC]$, that is, $d(\big[[RC, RC], [RC, RC]\big])=0$, a contradiction.

Choose $u\in [RC, RC]\setminus C$ such that $d(u^2)\ne 0$ and let $K:=Cu+C$. Then $K$ is a noncentral abelian Lie ideal of $RC$ such that $d(K)\nsubseteq K$.

Write $[RC, RC]=Cu\oplus Cv\oplus C$. Note that $[u, v]\ne0$, otherwise $v\in Cu+C$ (see \cite[Lemma 5.1]{lee2025}). We may assume that $[u, v]=1$.
Since $d$ is X-outer, $d(\beta)\ne 0$ for some $\beta\in C$. In particular, $C$ is an infinite field and
$$
C^{(2)}:=\{\alpha^2\mid \alpha\in C\}
$$
is also an infinite field. We claim that there exist infinitely many $\alpha_j\in C^{(2)}$ such that $d((u+\beta\alpha_jv)^2)\ne 0$.
Indeed, $d(\alpha_j)=0$ for all $j$ and so
\begin{eqnarray*}
 &&d((u+\beta\alpha_jv)^2)\\
  &=&d(u^2)+d(\beta\alpha_j[u, v])+\beta^2\alpha_j^2d(v^2)\\
  &=&d(u^2)+\beta\alpha_jd([u, v])+\alpha_jd(\beta)[u, v]+\beta^2\alpha_j^2d(v^2)\\
  &=&d(u^2)+\alpha_jd(\beta)+\beta^2\alpha_j^2d(v^2)\ne 0
  \end{eqnarray*}
for infinitely many $\alpha_j\in C^{(2)}$. This proves that $d((u+\beta\alpha_jv)^2)\ne 0$ for some $0\ne \alpha_j\in C^{(2)}$.
Let $L_j:=C(u+\beta\alpha_jv)+C$ for all $j$. Then every $L_j$ is a noncentral abelian Lie ideal of $R$ satisfying $d(L_j)\nsubseteq L_j$ for any $j$.
Moreover, for $i\ne j$, we have
$$
[L_i, L_j]=\big[C(u+\beta\alpha_iv)+C, C(u+\beta\alpha_jv)+C\big]=C\beta(\alpha_i+\alpha_j)\ne 0.
$$
Thus all $L_i$'s are distinct (see \cite[Theorem 6.3]{lee2025}).
\end{proof}

We now prove the following theorem, i.e., Theorem E in the introduction.

\begin{thm}\label{thm29}
Let $\delta, d$ be nonzero derivations of $R$, and let $I$ be a nonzero ideal of $R$.

(i)\ If $d$ is X-inner and $\delta$ is X-outer, then $\delta d([I, I])\subseteq Z(R)$ iff $d([R, R])\subseteq Z(R)$;

(ii)\ If both $d$ and $\delta$ are X-inner, then $\delta d([I, I])\subseteq Z(R)$ iff either $d([R, R])\subseteq Z(R)$ or $\delta([R, R])\subseteq Z(R)$;

(iii)\ If $d$ is X-outer and $\delta$ is X-inner, then $\delta d([I, I])\subseteq Z(R)$ iff $\delta([R, R])\subseteq Z(R)$;

(iv)\ If both $\delta$ and $d$ are X-outer, then $\delta d([I, I])\subseteq Z(R)$ iff there exists $\beta\in Z(R)$ such that neither $d(\beta)=0$ nor $\delta(\beta)=0$, and
\begin{eqnarray}
\delta(\beta)d+d(\beta)\delta=\text{\rm ad}_g\ \text{\rm and}\ \ d^2=\mu d+ \text{\rm ad}_h
\label{eq:11}
\end{eqnarray}
for some $g\in RC$, $h\in [RC, RC]$ and $\mu\in C$ such that $\mu=0$ iff $g\in [RC, RC]$, and if $\mu\ne 0$ then
$
g^2+\delta(\beta)\mu g\in C.
$
\end{thm}

\begin{proof}
Since $I$ and $RC$ satisfy the same differential identities
(see \cite[Theorem 2]{lee1992}), it follows that $\delta d([I, I])\subseteq Z(R)$ iff $\delta d([RC, RC])\subseteq C$, and $d([R, R])\subseteq Z(R)$
iff $d([RC, RC])\subseteq C$.

In view of Lemma \ref{lem18}, there exist noncentral abelian Lie ideals $K, L$ of $RC$ such that
$$
[RC, RC]=K+L\ \text{\rm and}\ K\cap L=C.
$$
Recall that $KC=K$ and $LC=L$. Moreover, we can choose $K, L$ such that neither $d(K)\subseteq K$ nor $d(L)\subseteq L$.
In this case, $\delta d([RC, RC])\subseteq C$ iff both $\delta d(K)\subseteq C$ and $\delta d(L)\subseteq C$.

By Lemma \ref{lem18}, (i), (ii) and (iii) follow directly from (i), (ii) and (iii) of Theorem \ref{thm34}, respectively.

Finally, we prove (iv).
For ``$\Rightarrow$", since $\delta d([I, I])\subseteq Z(R)$, we get $\delta d([RC, RC])\subseteq C$. Thus both $\delta d(K)\subseteq C$ and $\delta d(L)\subseteq C$.
 In view of Theorem \ref{thm34}, there exists $\beta\in Z(R)$ such that neither $d(\beta)=0$ nor $\delta(\beta)=0$, and there exist $g\in RC$, $\mu\in C$, and $h\in [RC, RC]$ such that
$$
\delta(\beta)d+d(\beta)\delta=\text{\rm ad}_g,\ d^2=\mu d+\text{\rm ad}_h
$$
and
\begin{eqnarray}
\delta(\beta)\mu z+[g, z]\in C\ \ \forall z\in K+L=[RC, RC].
\label{eq:221}
\end{eqnarray}
Suppose that $\mu=0$. Then, by Eq.\eqref{eq:221}, $\big[g, [RC, RC]\big]\subseteq C$, implying that $g\in  [RC, RC]$. Conversely, assume that $g\in  [RC, RC]$. Then $[g, z]\in C$ for $z\in [RC, RC]$. Hence, by Eq.\eqref{eq:221} again, $\delta(\beta)\mu [RC, RC]\subseteq C$ and so $\mu=0$. This proves that $\mu=0$ iff $g\in [RC, RC]$.

Suppose that $\mu\ne 0$. Then $g\notin [RC, RC]$.
Since neither $d(L)\subseteq L$ nor $g\in [RC, RC]$, it follows from Theorem \ref{thm34} that $g^2+\delta(\beta)\mu g\in L$. Similarly, since $d(K)\nsubseteq K$, we get $g^2+\delta(\beta)\mu g\in K$. By the fact that $K\cap L=C$, we conclude that $g^2+\delta(\beta)\mu g\in C$, as desired.

For ``$\Leftarrow$",  for $a, b\in RC$, we let $a\equiv b$ stand for $a-b\in C$. Let $z\in [RC, RC]$.
Then $d(z)\in [RC, RC]$ and $[h, z]\in C$ as $h\in [RC, RC]$. In view of Eq.\eqref{eq:11}, we have
\begin{eqnarray}
d(\beta)\delta d(z)&=&\delta(\beta)d^2(z)+[g, d(z)]\nonumber\\
                                   &=&\delta(\beta)\big(\mu d(z)+[h, z]\big)+[g, d(z)]\nonumber\\
                                   &\equiv& \delta(\beta)\mu d(z)+[g, d(z)].
\label{eq:320}
\end{eqnarray}
for all $z\in  [RC, RC]$.

Case 1: $\mu=0$. Then $g\in [RC, RC]$ and so $[g, d(z)]\in C$ for $z\in [RC, RC]$.  By Eq.\eqref{eq:320}, $\delta d(z)\in C$ for all $z\in [RC, RC]$,
as desired.

Case 2: $\mu\ne 0$. Then $g\notin [RC, RC]$ and $g^2+\delta(\beta)\mu g\in C$. Commuting Eq.\eqref{eq:320} with $g$, we get
\begin{eqnarray}
d(\beta)[g, \delta d(z)]&=&\big[g, \delta(\beta)\mu d(z)+[g, d(z)]\big]\nonumber\\
                                          &=&\big[g^2+\delta(\beta)\mu g, d(z)\big]\nonumber\\
                                          &=&0
\label{eq:322}
\end{eqnarray}
for all $z\in [RC, RC]$. Since $g\notin [RC, RC]$, we get $RC=Cg+[RC, RC]$. Note that
$
\big[[RC, RC], \delta d(z)\big]\in C
$
for $z\in [RC, RC]$. Then
$$
\big[RC, \delta d(z)\big]\subseteq \big[Cg, \delta d(z)\big]+\big[[RC, RC], \delta d(z)\big]\subseteq C
$$
for all $z\in [RC, RC]$. So $\delta d(z)\in C$ for all $z\in [RC, RC]$ (see Lemma \ref{lem8} (ii)).
\end{proof}

\begin{thm}\label{thm35}
Let $\delta, d$ be derivations of $R$.

(i)\ If $d$ is X-inner and $\delta$ is X-outer, then $\delta d(R)\subseteq Z(R)$ iff $d=0$;

(ii)\ If $d$ and $\delta$ are X-inner, then $\delta d(R)\subseteq Z(R)$ iff either $d=0$ or $\delta([R, R])\subseteq Z(R)$;

(iii)\ If $d$ is X-outer and $\delta$ is X-inner, then $\delta d(R)\subseteq Z(R)$ iff $\delta=0$;

(iv)\ If both $\delta$ and $d$ are X-outer, then $\delta d(R)\subseteq Z(R)$ iff there exists $\beta\in Z(R)$ such that neither $d(\beta)=0$ nor $\delta(\beta)=0$, $\delta(\beta)d=d(\beta)\delta$, and  $d^2=0$.
\end{thm}

\begin{proof}
(i)\ Assume that $\delta d(R)\subseteq Z(R)$.
 Write $d=\text{\rm ad}_a$ for some $a\in RC$. Then
$$
\delta d(x)=\delta([a, x])=[\delta(a), x]+[a, \delta(x)]\in C
$$
for all $x\in R$. That is, $\ad_z\,\ad_{\delta(a)}+\ad_z\,\ad_a\, \delta=0$ for all $z\in RC$.
Since $\delta$ is X-outer, applying Proposition \ref{pro3} (i) we get $\ad_z\,\ad_a=0$ for all $z\in RC$, i.e.,
$\big[RC, [a, RC]\big]=0$. By Lemma \ref{lem8} (ii), $a\in C$, that is, $d=0$ follows.

(ii)\ Write $\delta=\text{\rm ad}_b$ and $d=\text{\rm ad}_a$ for some $a, b\in RC$.
Then
\begin{eqnarray*}
\delta d(R)\subseteq Z(R) &\Leftrightarrow& \big[b, [a, RC]\big]\subseteq C\\
                                                 &\Leftrightarrow&  \text{\rm either}\ a\in C\ \text{\rm or}\ b\in [RC, RC]\ (\text{\rm by Theorem \ref{thm28}})\\
                                                  &\Leftrightarrow& \text{\rm either}\ d=0\ \text{\rm or}\ \delta([R, R])\subseteq Z(R)\ (\text{\rm by Lemma \ref{lem14}}).
 \end{eqnarray*}

(iii)\ Write $\delta=\text{\rm ad}_b$ for some $b\in RC$.

``$\Rightarrow$":\ Assume that $\delta d(R)\subseteq Z(R)$, that is, $\big[b, d(x)\big]\in C$ for all $x\in R$ and hence for all $x\in RC$.
That is, $\ad_y\,\ad_bd=0$ for all $y\in RC$. In view of Proposition \ref{pro3} (i), $\ad_y\,\ad_b=0$ for all $y\in RC$. That is, $[y, [b, RC]]=0$ for all $y\in RC$, implying that $b\in C$ (see Lemma \ref{lem8} (ii)), i.e., $\delta=0$.
The part ``$\Leftarrow$" is trivial.

(iv)\ ``$\Rightarrow$":\ Assume that $\delta d(R)\subseteq Z(R)$. In particular, $\delta d([R, R])\subseteq Z(R)$. In view of Theorem \ref{thm29} (iv),
there exists $\beta\in Z(R)$ such that neither $d(\beta)=0$ nor $\delta(\beta)=0$, and
\begin{eqnarray}
\delta(\beta)d+d(\beta)\delta=\text{\rm ad}_g\ \text{\rm and}\ \ d^2=\mu d+ \text{\rm ad}_h
\label{eq:22}
\end{eqnarray}
for some $g\in RC$, $h\in [RC, RC]$ and $\mu\in C$ such that $\mu=0$ iff $g\in [RC, RC]$, and if $\mu\ne 0$ then
$g^2+\delta(\beta)\mu g\in C$.

Case 1:\ $\mu=0$. Then $d^2=\text{\rm ad}_h$ and $g\in [RC, RC]$. Let $x\in R$. Applying Eq.\eqref{eq:22}, we have
$
\delta(\beta)d^2(x)+d(\beta)\delta d(x)=[g, d(x)],
$
that is,
$
\delta(\beta)[h, x]+d(\beta)\delta d(x)=[g, d(x)].
$
Hence
$
[g, d(x)]+[\delta(\beta)h, x]=d(\beta)\delta d(x)\in C.
$
That is,
$$
\ad_z\ad_gd+\ad_z\,\ad_{\delta(\beta)h}=0
$$
for all $x, z\in RC$.
Since $d$ is X-outer, applying Proposition \ref{pro3} (i) we get
\begin{eqnarray}
\ad_z\ad_g=0\ \text{\rm and}\ \ \ad_z\,\ad_{\delta(\beta)h}=0
\label{eq:25}
\end{eqnarray}
for all $z\in RC$. The second part of Eq.\eqref{eq:25} implies that $h\in C$. Thus $d^2=\mu d+ \text{\rm ad}_h=0$.
The first part of Eq.\eqref{eq:25} implies that $[g, RC]\subseteq C$ and so $g\in C$. By Eq.\eqref{eq:22}, we get $\delta(\beta)d=d(\beta)\delta$.

Case 2:\ $\mu\ne 0$. We have
$
g^2+\delta(\beta)\mu g\in C.
$
Let $x\in RC$. By Eq.\eqref{eq:22}, we have
$$
\delta(\beta)d^2(x)+d(\beta)\delta d(x)=\big[g, d(x)\big],
$$
that is,
$
\delta(\beta)\big(\mu d(x)+[h, x]\big)+\big[g, d(x)\big]=d(\beta)\delta d(x)\in C,
$
i.e.,
\begin{eqnarray}
\ad_z \big(\delta(\beta)\mu+\ad_g\big)d+\ad_z\,\ad_{\delta(\beta)h}=0
\label{eq:29}
\end{eqnarray}
for all $z\in RC$.
Since $d$ is X-outer,
applying Proposition \ref{pro3} (i) to Eq.\eqref{eq:29}, we have
 \begin{eqnarray}
\ad_z \big(\delta(\beta)\mu+\ad_g\big)=0\ \text{\rm and}\ \ \ad_z\,\ad_{\delta(\beta)h}=0
\label{eq:27}
\end{eqnarray}
for all $z\in RC$. By Eq.\eqref{eq:27} we have $\delta(\beta)[h, RC]\subseteq C$, implying $h\in C$.
So, by Eq.\eqref{eq:22}, $d^2=\mu d$. By Eq.\eqref{eq:27} again, we have
\begin{eqnarray}
\delta(\beta)\mu x+\big[g, x\big]\in C.
\label{eq:28}
\end{eqnarray}
for all $x\in RC$.  Replacing $x$ by $x+g$ in Eq.\eqref{eq:28}, we have $\delta(\beta)\mu (x+g)+\big[g, x\big]\in C$ for all $x\in RC$.
Together with Eq.\eqref{eq:28} we conclude that $\delta(\beta)\mu g\in C$ and so $g\in C$. By  Eq.\eqref{eq:22}, we have
$\delta(\beta)d=d(\beta)\delta$. Hence, let $x\in RC$, we have
$$
\delta(\beta)\mu d(x)=\delta(\beta)d^2(x)=d(\beta)\delta d(x)\in C
$$
and so $\delta(\beta)\mu d(x)\in C$. That is, $\delta(\beta)\mu d(RC)\subseteq C$. This implies that $\mu=0$, a contradiction.

The part ``$\Leftarrow$" is trivial.
\end{proof}

\end{document}